\documentclass[10pt]{amsart}
\usepackage{amsmath, amssymb, amsthm, amscd, graphicx}

\theoremstyle{plain}
\newtheorem{prop}{Proposition}[section]

\newtheorem*{propS}{Proposition}
\newtheorem*{claim}{Claim}
\newtheorem{conj}[prop]{Conjecture}
\newtheorem{coro}[prop]{Corollary}
\newtheorem{lemm}[prop]{Lemma}
\newtheorem{ques}[prop]{Question}
\newtheorem{thrm}[prop]{Theorem}

\theoremstyle{definition}
\newtheorem{defi}[prop]{Definition}
\newtheorem*{nota}{Notation}
\newtheorem*{conv}{Convention}
\newtheorem{exam}[prop]{Example}
\newtheorem{rema}[prop]{Remark}
\newtheorem*{rema*}{Remark}
\newtheorem*{ackn*}{Acknowledgement}

\numberwithin{equation}{section}

\newcommand\0{0}
\newcommand\1{1}

\renewcommand\a{\alpha}
\renewcommand\aa{a}
\renewcommand\AA{A}
\newcommand\aab{\overline{a}}
\newcommand\act{\mathbin{\scriptscriptstyle\bullet}}
\newcommand\Act{\alpha}

\newcommand\Add{\mathbb{A}}
\newcommand\antiexp{\mathrel{{}_{\scriptscriptstyle\LD}\!\!\leftarrow}}

\renewcommand\b{\beta}
\newcommand\bb{b}
\newcommand\BB{B}
\newcommand\bbb{\overline{b}}
\newcommand\BBB{\mathcal{B}^{\scriptscriptstyle+}}
\newcommand\BBBB{\widehat{\mathcal{B}}^{\scriptscriptstyle+}}

\newcommand\BP[1]{B^{\scriptscriptstyle+}_{#1}}

\newcommand\CA{\mathcal{A}^{\scriptscriptstyle+}}
\newcommand\CAct[2]{\mathcal{C}(#1,#2)}
\newcommand\CAm{\mathcal{A}^{\scriptscriptstyle+}_0}
\newcommand\cc{c}
\newcommand\CC{C}
\newcommand\ccb{\overline{c}}
\newcommand\CCb{\overline{C}}
\newcommand\CCbh{\widehat{\overline{C}}}
\newcommand\CCC{\mathcal{C}}
\newcommand\CCCC{\widehat{\mathcal{C}}}
\newcommand\CCh{\widehat{C}}
\newcommand\CLD{\mathcal{LD}^{\scriptscriptstyle+}}

\newcommand\CLDm{\mathcal{LD}^{\scriptscriptstyle+}_0}
\newcommand\CMX{\CAct\MM\XX}

\renewcommand\d{\delta}
\newcommand\D{\Delta}

\newcommand\dd{d}
\newcommand\DD[1]{D_{#1}}

\newcommand\Div{\mathrm{Div}}
\renewcommand\div{\prec}
\newcommand\dive{\preccurlyeq}

\newcommand\ea{\epsilon}
\newcommand\ee{e}
\newcommand\eLD{=_{\scriptscriptstyle\LD}}
\newcommand\eval{\mathrm{eval}}
\renewcommand\exp{\rightarrow_{\scriptscriptstyle\LD}}
\newcommand\ew{\epsilon}

\newcommand\f{\phi}
\newcommand\ff{f}

\newcommand\Fp{F^{\scriptscriptstyle+}}

\newcommand\g{\gamma}
\newcommand\Gar{\mathsf{L\!G}}
\renewcommand\gcd{\mathrm{gcd}}
\let\ge=\geqslant
\renewcommand\gg{g}

\newcommand\head[1]{H(#1)}
\newcommand\Head[2]{H_{#1}(#2)}
\newcommand\hh{h}
\renewcommand\hom{\mathrm{Hom}}
\newcommand\Hom{\mathcal{H}om}

\newcommand\id[1]{1_{{#1}}}
\newcommand\ID{\mathrm{id}}
\newcommand\ie{\emph{i.e.}}
\newcommand\ii{i}
\newcommand\inv{^{-1}}

\newcommand\jj{j}

\newcommand\kk{k}

\newcommand\lcm{\mathrm{lcm}}
\newcommand\LD{\mathrm{LD}}
\let\le=\leqslant
\newcommand\Ldots{...\,}
\renewcommand\lg{\mathrm{lg}}

\newcommand\LG{(\Gar_0)}
\newcommand\LGi{(\Gar_1)}
\newcommand\LGii{(\Gar_2)}
\newcommand\LGiii{(\Gar_3)}
\newcommand\LGloc{(\Gar_3^{\ell oc})}
\newcommand\LGG{(\Gar_0)}
\newcommand\LGGi{(\Gar_1)}
\newcommand\LGGii{(\Gar_2)}
\newcommand\LGGiii{(\Gar_3)}

\newcommand\MA{\mathrm{A}^{\scriptscriptstyle+}}
\newcommand\MAct[2]{[#1,#2]}

\newcommand\MM{M}
\newcommand\MMb{\overline{M}}
\newcommand\MMM{M'}
\newcommand\MLD{\mathrm{LD}^{\scriptscriptstyle+}}
\newcommand\Mon[1]{\langle#1\rangle^{\scriptscriptstyle+}}
\newcommand{\multe}{\mathrel{\widetilde{%
\raisebox{6pt}{\hspace{1pt}}\hspace{-1.5pt}\smash\succcurlyeq}}}

\newcommand\Nat{\mathbb{N}}
\newcommand\nn{n}

\newcommand\nno{{n-1}}

\newcommand\Obj{\mathcal{O}bj}
\newcommand\op{\star}
\newcommand\OP{\mathbin{\hbox to
1.8ex{$\odot\hspace{-1.45ex}\star$}}}

\newcommand\PI{\widehat\pi}
\newcommand\pio{\varpi}
\newcommand\pp{p}

\newcommand\qq{q}
\newcommand\quot{\mathord{/\!\!}}

\newcommand\RH{\mathrm{ht}}
\newcommand\RLD{R_{\LD}}

\newcommand\RR{R}

\newcommand\sh[1]{\mathrm{sh}_{#1}}
\newcommand\sig[1]{\sigma_{\hspace{-0.2ex}#1}^{\null}}
\newcommand\Sig{\Sigma}
\newcommand\Sim{\mathcal{H}om^{sp}}
\newcommand\smallplus{{\scriptscriptstyle+}}

\newcommand\smallR{{\scriptscriptstyle R}}
\newcommand\source[1]{\partial_0#1}
\renewcommand\ss{s}
\renewcommand{\SS}{S}

\newcommand\SSb{\overline{S}}
\newcommand{\SSS}{\mathcal{S}}
\newcommand\sub[2]{#1\hspace{-0.3ex}_{/\hspace{-0.2ex}#2}}

\newcommand\target[1]{\partial_1#1}

\renewcommand\tt{t}
\newcommand\TT{T}

\newcommand\under{\backslash}
\newcommand\uu{u}

\newcommand\var{\mathrm{var}_{\!\smallR}}
\newcommand\vv{v}

\newcommand\ww{w}

\newcommand\xx{x}
\newcommand\XX{X}
\newcommand\xxb{{\overline\xx}}
\newcommand\XXb{\overline{\XX}}
\newcommand\XXX{X'}

\newcommand\yy{y}

\newcommand\zz{z}
\newcommand\ZZZ{\mathbb{Z}}

\begin{document}

\hfill{\tiny 2008-10}

\author{Patrick DEHORNOY}
\address{Laboratoire de Math\'ematiques Nicolas Oresme,
Universit\'e de Caen, 14032 Caen, France}
\email{dehornoy@math.unicaen.fr}
\urladdr{//www.math.unicaen.fr/\!\hbox{$\sim$}dehornoy}

\title{Left-Garside categories, self-distributivity, and braids}



\begin{abstract}
In connection with the emerging theory of Garside categories, we develop the notions of a left-Garside category
and of a locally left-Garside monoid. In this framework, the connection between the self-distributivity law~LD
and braids amounts to the result that a certain category associated with~LD is a left-Garside category, which
projects onto the standard Garside category of braids. This approach leads to a realistic program for
establishing the Embedding Conjecture of [Dehornoy, Braids and Self-distributivity, Birkha\"user (2000),
Chap.~IX].
\end{abstract}

\maketitle

The notion of a Garside monoid emerged at the end of the 1990's \cite{Dfx, Dgk} as a development of
Garside's theory of braids~\cite{Gar}, and it led to many developments \cite[...]{BesD, BesF, BeC,
BGG1, BGG2, BGG3, CMW, ChM, CrP, FrG, Geb, God, Lee, LeL, McC, Pic, Sib}. More recently, Bessis \cite{Bes},
Digne--Michel \cite{DiM}, and Krammer~\cite{Kra} introduced the notion of a Garside category as a further
extension, and they used it to capture new, nontrivial examples and improve our understanding of their
algebraic structure. The concept of a Garside category is also implicit in~\cite{DeL} and~\cite{GodT}, and
maybe in the many diagrams of~\cite{Dgd}.

Here we shall describe and investigate a new example of (left)-Garside category, namely a certain
category~$\CLD$ associated with the left self-distributivity law
\begin{equation}
\label{E:LD}\tag{LD}
\xx(\yy\zz) = (\xx\yy)(\xx\zz).
\end{equation}
The interest in this law originated in the discovery of several nontrivial structures
that obey it, in particular in set theory \cite{Dem, Lvb} and in low-dimensional topology~\cite{Joy, FeR, Mat}.
A rather extensive theory was developed in the decade 1985-95~\cite{Dgd}. 

Investigating self-distributivity in the light of Garside categories seems to be a good idea. It turns out 
that a large part of the theory developed so far can be summarized into one single statement, namely
\begin{quote}
\emph{The category~$\CLD$ is a left-Garside category,}
\end{quote}
stated as the first part of Theorem~\ref{T:Main}.

The interest of the approach should be at least triple. First, it gives an opportunity to
restate a number of previously unrelated properties  in a new language that is 
natural and should make them more easily understandable---this is probably not useless. In
particular, the connection between self-distributivity and braids is now expressed in the simple statement:
\begin{quote}
\emph{There exists a right-lcm preserving surjective functor of~$\CLD$ to the
Garside category of positive braids,}
\end{quote}
(second part of Theorem~\ref{T:Main}).
This result allows one to recover most of the usual algebraic properties of braids as a direct application of
the properties of~$\CLD$: roughly speaking, Garside's theory of braids is the emerged part if
an iceberg, namely the algebraic theory of self-distributivity.

Second, a direct outcome of the current approach is  a realistic program
for establishing the Embedding Conjecture. The latter is the most puzzling open question involving
free self-distributive systems. Among others, it says that the equivalence class of any bracketed
expression under self-distributivity is a semilattice, \ie, any two expressions admit a least upper bound with
respect to a certain partial ordering. Many equivalent forms of the conjecture are know~\cite[Chapter IX]{Dgd}.
At the moment, no complete proof is known, but we establish the following new result
\begin{quote}
\emph{Unless the left-Garside category~$\CLD$ is not regular, the Embedding Conjecture is true,}
\end{quote}
(Theorem~\ref{T:Embedding}). This result reduces a possible proof of the conjecture to a (long)
sequence of verifications.

Third, the category~$\CLD$ seems to be a seminal example of a left-Garside category, quite
different from all previously known examples of Garside categories. In particular, being strongly asymmetric,
$\CLD$ is not a Garside category. The interest of investigating such objects \emph{per se} is not
obvious, but the existence of a nontrivial example such as~$\CLD$ seems to be a good reason, and a help for
orientating further reaserch.   In particular, our approach emphasizes the role of \emph{locally left-Garside
monoids}\footnote{This is not the notion of a locally Garside monoid in the sense of~\cite{DiM}; we think that
the name ``preGarside'' is more relevant for that notion, which involves no counterpart of any Garside element
or map, but is only the common substratum of all Garside structures.}: this is a monoid~$\MM$ that fails to be
Garside because no global element~$\D$ exists, but nevertheless possesses a family of
elements~$\D_\xx$ that locally play the role of the Garside element and are indexed by a set on which
the monoid~$\MM$ partially acts. Most of the properties of left-Garside monoids extend to locally left-Garside
monoids, in particular the existence of least common multiples and, in good cases, of the greedy normal form.

\begin{ackn*}
Our definition of a left-Garside category is borrowed from~\cite{DiM} (up to a slight change in the formal
setting, see Remark~\ref{R:Equivalence}). Several proofs in Section~\ref{S:Simple} and~\ref{S:Normal} use
arguments that are already present, in one form or another, in~\cite{Adj, Thu, ElM, Eps, Cha, Dgk, GodT} and
now belong to folklore. Most appear in the unpublished paper~\cite{DiM} by  Digne and Michel, and are
implicit in Krammer's paper~\cite{Kra}. Our reasons for including such arguments here is that adapting them
to the current weak context requires some polishing, and that it makes it natural to introduce our two main
new notions, namely locally Garside monoids and regular left-Garside categories.
\end{ackn*}

The paper is organized in two parts. The first one (Sections~\ref{S:Cat} to~\ref{S:Normal}) contains those
general results about left-Garside categories and locally left-Garside monoids that will be needed
in the sequel, in particular the construction and properties of the greedy normal form. The second part
(Sections~\ref{S:LD} to~\ref{S:Intermediate}) deals with the specific case of the category~$\CLD$ and its
connection with braids. Sections~\ref{S:LD} and~\ref{S:CLD} review basic facts about the
self-distributivity law and explain the construction of the category~$\CLD$. Section~\ref{S:Main} is devoted to
proving that $\CLD$ is a left-Garside category and to showing how the results of
Section~\ref{S:Normal} might lead to a proof of the Embedding Conjecture. In Section~\ref{S:Reproving}, we
show how to recover the classical algebraic properties of braids from those of~$\CLD$. Finally, we explain in
Section~\ref{S:Intermediate} some alternative solutions for projecting~$\CLD$ to braids. In an appendix,
we briefly describe what happens when the associativity law replaces the self-distributivity law: here also a
left-Garside category appears, but a trivial one.

We use~$\Nat$ for the set of all positive integers.

\section{Left-Garside categories}
\label{S:Cat}

We define left-Garside categories and describe a uniform way of constructing such categories from
so-called locally left-Garside monoids, which are monoids with a convenient partial action.

\subsection{Left-Garside monoids}
\label{S:Monoids}

Let us start from the now classical notion of a Garside monoid. Essentially, a Garside monoid is a monoid in
which divisibility has good properties, and, in addition, there exists a distinguished element~$\D$
whose divisors encode the whole structure. Slightly different versions have been considered~\cite{Dfx, Dgk,
Dig}, the one stated below now being the most frequently used. In this paper, we are interested 
in one-sided versions involving left-divisibility only, hence we shall first introduce the notion of a
\emph{left-Garside monoid}.

Throughout the paper, if $\aa, \bb$ are elements of a monoid---or, from Section~\ref{S:GarCat}, morphisms of
a category---we say that $\aa$ \emph{left-divides}~$\bb$, denoted $\aa \dive \bb$, if there exists~$\cc$ 
satisfying $\aa\cc = \bb$. The set of all left-divisors of~$\aa$ is denoted by~$\Div(\aa)$. If $\aa\cc = \bb$
holds with $\cc \not= 1$, we say that $\aa$ is a
\emph{proper} left-divisor of~$\bb$,  denoted $\aa \div \bb$. 

We shall always consider monoids~$\MM$ where $1$ is the only invertible element, which will imply that
$\dive$ is a partial ordering. If two elements~$\aa, \bb$ of~$\MM$ admit a greatest lower bound~$\cc$ with
respect to~$\dive$, the latter is called a \emph{greatest common left-divisor}, or \emph{left-gcd}, of~$\aa$
and~$\bb$, denoted $\cc = \gcd(\aa,\bb)$. Similarly, a $\dive$-least upper bound~$\dd$ is called a
\emph{least common right-multiple}, or \emph{right-lcm},  of~$\aa$ and~$\bb$, denoted $\dd =
\lcm(\aa,\bb)$. We say that~$\MM$ \emph{admits local right-lcm's} if any two elements of~$\MM$ that admit
a common right-multiple admit a right-lcm. 

Finally, if $\MM$ is a monoid and $\SS, \SS'$ are subsets of~$\MM$, we say that $\SS$
\emph{left-generates}~$\SS'$ if every nontrivial element of~$\SS'$ admits at least one nontrivial left-divisor
belonging to~$\SS$.

\begin{defi}
\label{D:Monoid}
We say that a monoid~$\MM$ is \emph{left-preGarside} if 

\smallskip

$\LG$ for each~$\aa$ in~$\MM$, every $\div$-increasing sequence in~$\Div(\aa)$ is finite,

$\LGi$ $\MM$ is left-cancellative,

$\LGii$ $\MM$ admits local right-lcm's.

\smallskip

\noindent An element~$\D$ of~$\MM$ is called a \emph{left-Garside element} if

\smallskip

$\LGiii$ {$\Div(\D)$ left-generates~$\MM$, and $\aa \dive \D$ implies $\D \dive \aa\D$.}

\smallskip

\noindent We say that $\MM$ is \emph{left-Garside} if it is left-preGarside and possesses at
least one left-Garside element.
\end{defi}

Using ``generates'' instead of ``left-generates'' in~$\LGiii$ would make no difference, by the following trivial
remark---but the assumption~$\LG$ is crucial, of course.

\begin{lemm}
\label{L:LeftGenerate}
Assume that $\MM$ is a monoid satisfying~$\LG$. Then every subset~$\SS$ left-generating~$\MM$
generates~$\MM$.
\end{lemm}

\begin{proof}
Let $\aa$ be a nontrivial element of~$\MM$. By definition there exist~$\aa_1\not=1$ in~$\SS$
and~$\aa'$ satisfying $\aa = \aa_1 \aa'$. If $\aa'$ is trivial, we are done. Otherwise, there
exist~$\aa_2\not=1$ in~$\SS$ and~$\aa''$ satisfying $\aa' = \aa_2 \aa''$, and so on. The sequence $1$,
$\aa_1$, $\aa_1\aa_2$, ... is $\div$-increasing and it lies in~$\Div(\aa)$, hence it must be finite, yielding
$\aa = \aa_1 ... \aa_\dd$ with $\aa_1, ..., \aa_\dd$ in~$\SS$. 
\end{proof}

Right-divisibility is defined symmetrically: $\aa$ \emph{right-divides}~$\bb$ if $\bb = \cc\aa$ holds for
some~$\cc$. Then the notion of a right-(pre)Garside monoid is defined by replacing
left-divisibility by right-divisibility and left-product by right-product in Definition~\ref{D:Monoid}.

\begin{defi}
\label{D:GarsideM}
We say that a monoid $\MM$ is \emph{Garside with Garside element~$\D$} if $\MM$ is both left-Garside
with left-Garside element~$\D$ and right-Garside with right-Garside element~$\D$.
\end{defi}

The equivalence of the above definition with that of~\cite{Dig} is easily checked. 
The seminal example of a Garside monoid is the braid monoid~$\BP\nn$ equipped with
Garside's fundamental braid~$\D_\nn$, see for instance~\cite{Gar, Eps}. Other classical examples are
free abelian monoids and, more generally, all spherical Artin--Tits monoids~\cite{BrS}, as
well as the so-called dual Artin--Tits monoids~\cite{BKL, Bes}. Every Garside monoid embeds in a group of
fractions, which is then called a \emph{Garside group}.

Let us mention that, if a monoid~$\MM$ is left-Garside, then mild
conditions imply that it is Garside: essentially, it is sufficient that $\MM$ is right-cancellative and that the left-
and right-divisors of the left-Garside element~$\D$ coincide~\cite{Dgk}.

\subsection{Left-Garside categories}
\label{S:GarCat}

Recently, it appeared that a number of results involving Garside monoids still make sense in a
wider context where categories replace monoids~\cite{Bes, DiM, Kra}. A category is similar to a monoid,
but the product of two elements is defined only
when the target of the first is the source of the second. In the case of Garside monoids, the main benefit of
considering categories is that it allows for relaxing the existence of the global Garside element~$\D$ into a
weaker, local version depending on the objects of the category, namely a map from the objects to the
morphisms.

We refer to~\cite{Mac} for some basic vocabulary about categories---we use very little of it here.

\begin{conv}
Throughout the paper, composition of morphisms is denoted by a multiplication on the right: $\ff\gg$ means
``$\ff$ then $\gg$''. If $\ff$ is a morphism, the source of~$\ff$ is denoted~$\source\ff$, and its
target is denoted~$\target\ff$. In all examples, we shall make the source and
target explicit: morphisms are triples $(\xx, \ff,
\yy)$ satisfying
$$\source{(\xx, \ff, \yy)} = \xx, \qquad \target{(\xx, \ff, \yy)} = \yy.$$
A morphism~$\ff$ is said to be nontrivial if $\ff \not=\id{\source\ff}$ holds.
\end{conv}

We extend to categories the terminology of divisibility. So, we say that a morphism~$\ff$
is a \emph{left-divisor} of a morphism~$\gg$, denoted~$\ff \dive \gg$, if there
exists~$\hh$ satisfying $\ff \hh = \gg$. If, in addition, $\hh$ can be assumed to be
nontrivial, we say that $\ff \div \gg$ holds. Note that $\ff \dive \gg$ implies $\source\ff =
\source\gg$. We denote by~$\Div(\ff)$ the collection of all left-divisors of~$\ff$. 

The following definition is equivalent to Definition~2.10 of~\cite{DiM} by F.\,Digne and J.\,Michel---see
Remark~\ref{R:Equivalence} below.

\begin{defi}
\label{D:GarCat}
We say that a category~$\CCC$ is \emph{left-preGarside} if 

\smallskip

$\LGG$ for each~$\ff$ in~$\Hom(\CCC)$, every $\div$-ascending sequence in~$\Div(\ff)$ is
finite,

$\LGGi$ $\Hom(\CCC)$ admits left-cancellation,

$\LGGii$ $\Hom(\CCC)$ admits local right-lcm's.

\smallskip

\noindent A map $\D: \Obj(\CCC) \to \Hom(\CCC)$ is called a
\emph{left-Garside map} if, for each object~$\xx$, we have $\source{\D(\xx)} = \xx$ and

\smallskip

$\LGiii$ $\D(\xx)$ left-generates~$\hom(\xx, -)$, and $\ff \dive \D(\xx)$ implies $\D(\xx) \dive
\ff \D(\target{\ff})$.

\smallskip

\noindent We say that $\CCC$ is \emph{left-Garside} if it is left-preGarside and
possesses at least one left-Garside map.
\end{defi}

\begin{exam}
\label{X:Monoid}
Assume that $\MM$ is a left-Garside monoid with left-Garside element~$\D$. One trivially obtains a
left-Garside category $\CCC(\MM)$ by putting
$$\Obj(\CCC(\MM)) = \{1\}, \quad
\Hom(\CCC(\MM)) = \{1\} \times \MM \times \{1\}, \quad
\D(1) = \D.$$
Another left-Garside category~$\CCCC_\MM$ can be attached with~$\MM$, namely
taking 
$$\Obj(\CCCC(\MM)) = \MM, \quad
\Hom(\CCCC(\MM)) = \{(\aa, \bb, \cc) \mid \aa\bb = \cc\,\},
\quad
\D(\aa) = \D.$$
It is natural to call $\CCCC(\MM)$ the \emph{Cayley category} of~$\MM$ since its graph is the Cayley graph
of~$\MM$ (defined provided $\MM$ is right-cancellative). 
\end{exam}

The notion of a right-Garside category can be defined symmetrically, exchanging left and right everywhere
and exchanging the roles of source and target. In particular, the map~$\D$ and Axiom~$\LGGiii$
is to be replaced by a map~$\nabla$ satisfying $\target{\nabla(\xx)} = \xx$ and, using $\bb \multe \aa$ for
``$\aa$ right-divides~$\bb$'',

\smallskip

$\LGiii$ $\nabla(\yy)$ right-generates~$\hom(-, \yy)$, and $\nabla(\yy) \multe \ff$ implies
$\nabla(\source\yy) \ff \multe \nabla(\yy)$.

\smallskip

\noindent Then comes the natural two-sided version of a Garside category~\cite{Bes, DiM}.

\begin{defi}
\label{D:Garside}
We say that a category~$\CCC$ is \emph{Garside with Garside map $\D$} if  $\CCC$ is left-Garside with
left-Garside map~$\D$ and right-Garside  with right-Garside map~$\nabla$ satisfying  $\D(\xx) =
\nabla(\target{(\D(\xx)})$ and $\nabla(\yy) = \D(\source{(\nabla(\yy)})$ for all objects~$\xx, \yy$.
\end{defi}

It is easily seen that, if $\MM$ is a Garside monoid, then the categories~$\CCC(\MM)$ and~$\CCCC(\MM)$ of
Example~\eqref{X:Monoid} are Garside categories. Insisting that the maps~$\D$ and~$\nabla$ involved in the
left- and right-Garside structures are connected as in Definition~\ref{D:Garside}
is crucial: see Appendix for a trivial example where the connection fails.

\subsection{Locally left-Garside monoids}
\label{S:Action}

We now describe a general method for constructing a left-Garside category starting from a
monoid equipped with a partial action on a set. The trivial examples of Example~\ref{X:Monoid} enter this family, and so do
the two categories~$\CLD$ and~$\BBB$ investigated in the second part of this paper. 

\begin{defi}
\label{D:Action}
Assume that $\MM$ is a monoid. We say that $\Act: \MM \times
\XX \to \XX$ is a \emph{partial (right) action} of~$\MM$ on~$\XX$ if, writing $\xx \act
\aa$ for $\Act(\aa)(\xx)$, 

$(i)$ $\xx \act 1 = \xx$ holds for each~$\xx$ in~$\XX$,

$(ii)$ $(\xx \act \aa) \act \bb = \xx \act \aa\bb$ holds for all~$\xx, \aa, \bb$, this meaning that either both
terms are defined and they are equal, or neither is defined,

$(iii)$ for each finite subset~$\SS$ in~$\MM$, there exists $\xx$ in~$\XX$ such that $\xx \act \aa$
is defined for each~$\aa$ in~$\SS$.

\noindent In the above context, for each~$\xx$ in~$\XX$, we put
\begin{equation}
\label{E:Mx}
\MM_\xx = \{\ \aa\in\MM \mid \mbox{ $\xx\act\aa$ is defined }\}.
\end{equation}
\end{defi}

Then Conditions $(i)$, $(ii)$, $(iii)$ of Definition~\ref{D:Action} imply
$$1 \in \MM_\xx, \quad \aa\bb\in\MM_\xx
\Leftrightarrow (\aa \in \MM_\xx \ \&\  \bb \in \MM_{\xx\act\aa}), \quad
\forall\mbox{\,finite\,}\SS \ \exists\xx(\SS
\subseteq\MM_\xx).$$

A monoid action in the standard sense, \ie, an everywhere defined action, is a partial action. For a more typical
case, consider the $\nn$-strand Artin braid monoid~$\BP\nn$. We recall that
$\BP\nn$ is defined  for $\nn \le \infty$ by the monoid presentation
\begin{equation}
\label{E:Braid}
\BP\nn = \bigg\langle \sig1,\Ldots, \sig{n-1} \ \bigg\vert\ 
\begin{matrix}
\sig\ii \sig j = \sig j \sig\ii 
&\text{for} &\vert \ii-\jj \vert\ge 2\\
\sig\ii \sig j \sig\ii = \sig j \sig\ii \sig j 
&\text{for} &\vert \ii-\jj \vert = 1
\end{matrix}
\ \bigg\rangle^{\!\!\smallplus}.
\end{equation}
Then we obtain a partial action of~$\BP\infty$ on~$\Nat$ by putting
\begin{equation}
\label{E:BraidAction}
\nn \act \aa = 
\begin{cases}
\ \ \nn
&\mbox{if $\aa$ belongs to~$\BP\nn$,}\\
\mbox{\ \ undefined}
&\mbox{otherwise.}
\end{cases}
\end{equation}

A natural category can then be associated with every partial action of a monoid.

\begin{defi}
\label{D:ActionCat}
For $\Act$ a partial action of a monoid~$\MM$ on a set~$\XX$, the
category \emph{associated with~$\Act$}, denoted~$\CCC(\Act)$, or $\CAct\MM\XX$ if the action is
clear, is defined by
$$\Obj(\CMX) = \XX, \quad
\Hom(\CMX) = \{ (\xx, \aa, \xx\act\aa) \mid \xx \in\XX, \aa\in \MM\}.$$
\end{defi}

\begin{exam}
\label{X:Braid}
We shall denote by~$\BBB$ the category associated with the action~\eqref{E:BraidAction} of~$\BP\infty$
on~$\Nat$, \ie, we put
$$\Obj(\BBB) = \Nat, \quad
\Hom(\BBB) = \{ (\nn, \aa, \nn) \mid \nn \in\BP\nn\}.$$
Define $\D : \Obj(\BBB) \to \Hom(\BBB)$ by $\D(\nn) = (\nn, \D_\nn, \nn)$.  Then the well known
fact that $\BP\nn$ is a Garside monoid for each~$\nn$~\cite{Gar, KaT} easily implies that $\BBB$ is a
Garside category (as will be formally proved in Proposition~\ref{P:Criterion} below). 
\end{exam}

The example of~$\BBB$ shows the benefit of going from a monoid to a category. The monoid~$\BP\infty$
is not a (left)-Garside monoid, because it is of infinite type and there cannot exist a global Garside
element~$\D$. However, the partial action of~\eqref{E:BraidAction} enable us to restrict to
subsets~$\BP\nn$ (submonoids in the current case) for which Garside elements exist: with the notation
of~\eqref{E:Mx}, $\BP\nn$ is~$(\BP\infty)_\nn$. Thus the categorical context allows to capture the fact that
$\BP\infty$ is, in some sense, \emph{locally} Garside. It is easy to formalize these ideas in a general setting.

\begin{defi}
\label{D:LocGarside}
Let $\MM$ be a monoid with a partial action~$\Act$ on a set~$\XX$. A sequence
$(\D_\xx)_{\xx\in\XX}$ of elements of~$\MM$ is called a \emph{left-Garside sequence}
for~$\Act$ if,  for each~$\xx$ in~$\XX$, the element $\xx\act\D_\xx$ is defined and

\smallskip

$\LGloc$ {$\Div(\D_\xx)$ left-generates $\MM_\xx$ and $\aa\dive\D_\xx $ implies $\D_\xx
\dive\aa\,\D_{\xx\act\aa}$.}

\smallskip

\noindent The monoid~$\MM$ is said to be \emph{locally left-Garside with respect to~$\Act$} if it is
left-preGarside and there is at least one left-Garside sequence for~$\Act$.
\end{defi}

A typical example of a locally left-Garside monoid is $\BP\infty$ with its
action~\eqref{E:BraidAction} on~$\Nat$. Indeed, the sequence $(\D_\nn)_{\nn \in \Nat}$ is clearly a
left-Garside sequence for~\eqref{E:BraidAction}.

The nexy result should appear quite natural.

\begin{prop}
\label{P:Criterion}
Assume that $\MM$ is a locally left-Garside monoid with left-Gars\-ide sequence $(\D_\xx)_{\xx
\in \XX}$. Then $\CMX$ is a left-Garside category with left-Garside map $\D$ defined by
$\D(\xx) = (\xx, \D_\xx, \xx\act\D_{\xx})$.
\end{prop}

\begin{proof}
By definition, $(\xx, \aa, \yy) \dive (\xx', \aa', \yy')$ in~$\CMX$ implies $\xx' =
\xx$ and $\aa \dive \aa'$ in~$\MM$. So the hypothesis that $\MM$ satisfies~$\LG$
implies that $\CMX$ does.

Next, $(\xx, \aa, \yy)(\yy, \bb, \zz) = (\xx, \aa, \yy)(\yy, \bb', \zz')$ implies $\aa\bb = \aa\bb'$
in~$\MM$, hence $\bb = \bb'$ by~$\LGi$, and, therefore, $\CMX$ satisfies~$\LGGi$.

Assume $(\xx, \aa, \yy)(\yy, \bb', \xx') = (\xx, \bb, \zz)(\zz, \aa', \xx')$ in~$\Hom(\CMX)$.
Then $\aa\bb' = \bb\aa'$ holds in~$\MM$. By~$\LGii$, $\aa$ and $\bb$ admit a
right-lcm~$\cc$, and we have $\aa \dive \cc$, $\bb \dive \cc$, and $\cc \dive \aa\bb'$.
By hypothesis, $\xx \act \aa\bb'$ is defined, hence so is $\xx \act \cc$, and it is obvious
to check that $(\xx, \cc, \xx\act\cc)$ is a right-lcm of~$(\xx, \aa, \yy)$
and~$(\xx,\bb,\zz)$ in~$\Hom(\CMX)$. Hence $\CMX$ satisfies~$\LGGii$. 

Assume that $(\xx, \aa, \yy)$ is a nontrivial morphism of~$\Hom(\CMX)$. This means that $\aa$ is
nontrivial, so, by~$\LGloc$, some left-divisor~$\aa'$ of~$\D_\xx$ is a left-divisor of~$\aa$. Then
$(\xx, \aa', \xx\act\aa') \dive \D(\xx)$ holds, and $\D(\xx)$ left-generates~$\hom(\xx, -)$.

Finally, assume $(\xx, \aa, \yy) \dive \D(\xx)$ in~$\Hom(\CMX)$. This implies $\aa \dive \D_\xx$
in~$\MM$.  Then $\LGloc$ in~$\MM$ implies $\D_\xx \dive \aa \D_\yy$. By hypothesis, $\yy \act
\D_\yy$ is defined, and we have
$ (\xx, \aa, \yy) \D(\yy)  = (\xx, \aa\D_\yy, \xx \act \aa\D_\yy)$, of which $(\xx, \D_\xx, \xx\act\D_\xx)$
is a left-divisor in~$\Hom(\CMX)$. So $\LGiii$ is satisfied in~$\CMX$.
\end{proof}

It is not hard to see that, conversely, if $\MM$ is a left-preGarside monoid, then $\CMX$ being a
left-Garside category implies that $\MM$ is a locally left-Garside monoid. We shall not use the result here.

If $\MM$ has a total action on~$\XX$, \ie, if $\xx\act\aa$ is defined for all~$\xx$ and~$\aa$, the
sets~$\MM_\xx$ coincide with~$\MM$, and Condition~$\LGloc$ reduces to~$\LGiii$. In this
case, each element~$\D_\xx$ is a left-Garside element in~$\MM$, and $\MM$ is a left-Garside monoid. A
similar result holds for each set~$\MM_\xx$ that turns out to be a submonoid (if any).

\begin{prop}
\label{P:Submonoid}
Assume that $\MM$ is a locally left-Garside monoid with left-Gars\-ide sequence $(\D_\xx)_{\xx\in\XX}$, and
$\xx$ is such that $\MM_\xx$ is closed under product and $\D_\yy  = \D_\xx$ holds for each~$\yy$
in~$\MM_\xx$. Then $\MM_\xx$ is a left-Garside submonoid of~$\MM$.
\end{prop}

\begin{proof}
By definition of a partial action, $\xx \act 1$ is defined, so
$\MM_\xx$ contains~$1$, and it is a submonoid of~$\MM$. We show that $\MM_\xx$ satisfies~$\LG, \LGi,
\LGii$, and~$\LGiii$. First, a counter-example to~$\LG$ in~$\MM_\xx$ would be a
counter-example to~$\LG$ in~$\MM$, hence $\MM_\xx$ satisfies~$\LG$. Similarly, an
equality $\aa\bb = \aa\bb'$ with $\bb \not= \bb'$ in~$\MM_\xx$ would also contradict~$\LGi$ in~$\MM$,
so $\MM_\xx$ satisfies~$\LGi$. Now,  assume that $\aa$ and $\bb$ admit in~$\MM_\xx$, hence in~$\MM$,
a common right-multiple~$\cc$. Then $\aa$ and $\bb$ admit a right-lcm~$\cc'$ in~$\MM$. By hypothesis,
$\xx \act \cc$ is defined, and $\cc' \dive \cc$ holds. By definition of a partial action, $\xx
\act \cc'$ is defined as well, \ie, $\cc'$ lies in~$\MM_\xx$, and it is a right-lcm of~$\aa$ and~$\bb$
in~$\MM_\xx$. So $\MM_\xx$ satisfies~$\LGii$, and it is left-preGarside. 

Next, $\D_\xx$ is a left-Garside element in~$\MM_\xx$. Indeed, let~$\aa$ be any nontrivial element
of~$\MM_\xx$. By~$\LGloc$, there exists a nontrivial divisor~$\aa'$ of~$\aa$ satisfying $\aa' \dive \D_\xx$.
By definition of a partial action, $\xx \act
\aa'$ is defined, so it belongs to~$\MM_\xx$, and $\D_\xx$ left-generates~$\MM_\xx$. Finally,
assume $\aa \dive \D_\xx$. As $\D_\xx$ belongs to~$\MM_\xx$, this implies $\aa \in
\MM_\xx$, hence $\D_\xx \dive \aa\D_{\xx\act\aa}$ by~$\LGloc$, \ie, $\D_\xx \dive \aa \D_\xx$ since we
assumed that the sequence~$\D$ is constant on~$\MM_\xx$. So $\D_\xx$  a left-Garside in~$\MM_\xx$.
\end{proof}

\goodbreak

\section{Simple morphisms}
\label{S:Simple}

We return to general left-Garside categories and establish a few basic results. 
As in the case of Garside monoids, an important role is played by the divisors of~$\D$, a local notion here.

\subsection{Simple morphisms and the functor~$\f$}

\begin{defi}
\label{D:Simple}
Assume that $\CCC$ is a left-Garside category. A morphism~$\ff$ of~$\CCC$ is called
\emph{simple} if it is a left-divisor of~$\D(\source\ff)$. In this case, we denote by~$\ff^*$
the unique simple morphism satisfying $\ff \, \ff^* = \D(\source\ff)$. The family of all simple morphisms
in~$\CCC$ is denoted by~$\Sim(\CCC)$.
\end{defi}

By definition, every identity morphism~$\id\xx$ is a left-divisor of every morphism with
source~$\xx$, hence in particular of~$\D(\xx)$. Therefore $\id\xx$ is simple.

\begin{defi}
\label{D:Phi}
Assume that $\CCC$ is a left-Garside category. We put
$\f(\xx) = \target{(\D(\xx))}$ for~$\xx$ in~$\Obj(\CCC)$, and $\f(\ff) =
\ff^{**}$ for~$\ff$ in~$\Sim(\CCC)$.
\end{defi}

Although straightforward, the following result is fundamental---and it is the main
argument for stating $\LGGiii$ in the way we did.

\begin{lemm}
\label{L:Simple}
Assume that $\CCC$ is a left-Garside category. 

$(i)$ If $\ff$ is a simple morphism, so are~$\ff^*$ and~$\f(\ff)$.

$(ii)$ Every right-divisor of a morphism~$\D(\xx)$ is simple.
\end{lemm}

\begin{proof}
$(i)$ By $\LGGiii$, we have $\ff \ff^* = \D(\source\ff) \dive \ff \D(\target\ff)$, hence $\ff^* \dive
\D(\target\ff)$ by left-cancelling~$\ff$. This shows that $\ff^*$ is simple. Applying the result to~$\ff^*$ shows
that $\f(\ff)$---as well as $\f^\kk(\ff)$ for each positive~$\kk$---is simple.

$(ii)$ Assume that $\gg$ is a right-divisor of~$\D(\xx)$. This means that there exists~$\ff$ satisfying $\ff \gg
= \D(\xx)$, hence $\gg = \ff^*$ by~$\LGi$. Then $\gg$ is simple by~$(i)$
\end{proof}

\begin{lemm}
\label{L:Basic}
Assume that $\CCC$ is a left-Garside category. 

$(i)$ The morphisms~$\id\xx$ are the only left- or right-invertible morphisms in~$\CCC$.

$(ii)$ Every morphism of~$\CCC$ is a product of simple morphisms.

$(iii)$ There is a unique
way to extend~$\f$ into a functor of~$\CCC$ into itself.

$(iv)$ The map~$\D$ is a natural
transformation of the identity functor into~$\f$, \ie, for each morphism~$\ff$, we have
\begin{equation}
\label{E:Phi}
\ff \, \D(\target\ff) = \D(\source\ff) \ \f(\ff).
\end{equation}
\end{lemm}

\begin{proof}
$(i)$ Assume $\ff \gg = \id\xx$ with $\ff \not= \id\xx$ and $\gg \not= \id{\target\ff}$.
Then we have
$$\id\xx \div \ff \div \ff\gg \div \ff \div \ff\gg \div ..., $$
an infinite $\div$-increasing sequence in~$\Div(\id\xx)$ that contradicts~$\LGG$.

$(ii)$ Let $\ff$ be a morphism of~$\CCC$, and let $\xx = \source\ff$. If $\ff$ is trivial,
then it is simple, as observed above. We wish to prove that simple morphisms generate~$\Hom(\CCC)$. Owing
to Lemma~\ref{L:LeftGenerate}, it is enough to prove that simple morphisms left-generate~$\Hom(\CCC)$, \ie,
that every nontrivial morphism with source~$\xx$ is left-divisible by a simple morphism with source~$\xx$,
in other words by a left-divisor of~$\D(\xx)$. This is exactly what the first part of Condition~$\LGiii$
claims.

$(iii)$ Up to now, $\f$ has been defined on objects, and on simple morphisms. Note that,
by construction,  \eqref{E:Phi} is satisfied for each simple morphism~$\ff$.
Indeed, applying Definition~\ref{D:Simple} for~$\ff$ and $\ff^*$ gives the relations
$$\ff \ff^* = \D(\source\ff)
\text{\quad and \quad}
\ff^* \ff^{**} = \D(\source{\ff^*}) = \D(\target\ff),$$
whence
$$\ff \D(\target\ff) = \ff \ff^* \ff^{**} = \D(\source\ff) \ff^{**} = \D(\source\ff)
\f(\ff).$$
Applying this to~$\ff = \id\xx$ gives $\D(\xx) = \D(\xx) \f(\id\xx)$, hence $\f(\id\xx) =
\id{\f(\xx)}$ by~$\LGGi$.

Let $\ff$ be an arbitrary morphism of~$\CCC$, and let $\ff_1 ... \ff_\pp$ and $\gg_1 ...
\gg_\qq$ be two decompositions of~$\ff$ as a product of simple morphisms, which exist by~$(ii)$.
Repeatedly applying~\eqref{E:Phi} to $\ff_\pp, ..., \ff_1$ and $\gg_\qq, ..., \gg_1$ gives
\begin{align*}
\ff \D(\target\ff) 
&= \ff_1 ... \ff_\pp \D(\target\ff) = \D(\source\ff) \f(\ff_1) ... \f(\ff_\pp)\\
&= \gg_1 ... \gg_\qq \D(\target\ff) = \D(\source\ff) \f(\gg_1) ... \f(\gg_\qq).
\end{align*}
By~$\LGGi$, we deduce $\f(\ff_1) ... \f(\ff_\pp) = \f(\gg_1) ... \f(\gg_\qq)$, and therefore there is no
ambiguity in defining $\f(\ff)$ to be the common value. In this way, $\f$ is extended to
all morphisms in such a way that $\f$ is a functor  and \eqref{E:Phi} always holds.
Conversely, the above definition is clearly the only one that extends~$\f$ into a functor.

$(iv)$ We have seen above that \eqref{E:Phi} holds for every morphism~$\ff$, so nothing new is needed here.
See Figure~\ref{F:NatTrans} for an illustration.
\end{proof}

\begin{figure}[htb]
\begin{picture}(38,18)(0,-1)
\put(8,0){\includegraphics{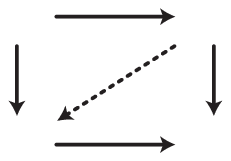}}
\put(0,7){$\D(\xx)$}
\put(30,7){$\D(\yy)$}
\put(5,0){$\f(\xx)$}
\put(26,0){$\f(\yy)$}
\put(15,-2.5){$\f(\ff)$}
\put(15.5,7.5){$\ff^*$}
\put(28,13){$\yy$}
\put(8,13){$\xx$}
\put(17,15.5){$\ff$}
\end{picture}
\caption{\sf\smaller Relation~\eqref{E:Phi}: the Garside map~$\D$ viewed as a natural transformation from
the identity functor to the functor~$\f$.}
\label{F:NatTrans}
\end{figure}

\begin{lemm}
\label{L:Dual}
Assume that $\CCC$ is a left-Garside category. Then, for each object~$\xx$
and each simple morphism~$\ff$, we have
\begin{equation}
\label{E:Dual}
\f(\D(\xx)) = \D(\f(\xx))
\mbox{\quad and \quad}
\f(\ff^*) = \f(\ff)^*.
\end{equation}
\end{lemm}

\begin{proof}
By definition, the source of~$\D(\xx)$ is~$\xx$ and its target is~$\f(\xx)$, hence
applying \eqref{E:Phi} with $\ff = \D(\xx)$ yields
$\D(\xx) \D(\f(\xx)) = \D(\xx) \f(\D(\xx))$, hence $\D(\f(\xx)) = \f(\D(\xx))$ after
left-cancelling~$\D(\xx)$.

On the other hand, let $\xx = \source\ff$. Then we have $\ff \ff^* = \D(\xx)$, and
$\source{(\f(\ff))} = \f(\xx)$. Applying~$\f$ and  the above relation, we find
$$\f(\ff) \f(\ff^*) = \f(\D(\xx)) = \D(\f(\xx)) = \D(\source{(\f(\ff))}) = \f(\ff) \f(\ff)^*.$$
Left-cancelling~$\f(\ff)$ yields $\f(\ff^*) = \f(\ff)^*$.
\end{proof}

\begin{rema}
\label{R:Equivalence}
We can now see that Definition~\ref{D:GarCat} is equivalent to Definition~2.10 of~\cite{DiM}: the only
difference is that, in the latter, the functor~$\f$ is part of the definition. Lemma~\ref{L:Basic}$(iv)$ shows
that a left-Garside category in our sense is a left-Garside category in the sense of~\cite{DiM}. Conversely, the
hypothesis that $\f$ and~$\D$ satisfy~\eqref{E:Phi} implies that, for $\ff : \xx \to \yy$, we have $\D(\xx)
\f(\yy) =  \ff \D(\yy)$, whence $\D(\xx)
\dive \ff \D(\yy)$ and $\ff^* \f(\yy) = \D(\yy)$, which, by~$\LGi$,
implies $\f(\ff) = \ff^{**}$. So every left-Garside category in the sense of~\cite{DiM} is a left-Garside category
in the sense of Definition~\ref{D:GarCat}.
\end{rema}

\subsection{The case of a locally left-Garside monoid}

We now consider the particular case of a category associated with a partial action of a
monoid~$\MM$.

\begin{lemm}
\label{L:Phi}
Assume that $\MM$ is a locally left-Garside monoid with left-Garside sequence $(\D_\xx)_{\xx\in\XX}$. Then
$\D_\xx \dive \aa\D_{\xx\act\aa}$ holds whenever $\xx \act
\aa$ is defined, and, defining $\f_\xx(\aa)$ by $\D_\xx \f_\xx(\aa) = \aa \D_{\xx\act\aa}$, we
have
\begin{equation}
\f(\xx) = \xx \act \D_\xx, \quad
\f((\xx, \aa, \yy)) = (\f(\xx), \f_\xx(\aa), \f(\yy)).
\end{equation}
\end{lemm}

\begin{proof}
Assume that $\xx \act \aa$ is defined. By Lemma~\ref{L:Basic}$(ii)$, the morphism~$(\xx, \aa, \xx\act\aa)$
of~$\CMX$ can be decomposed into a finite product of simple morphisms $(\xx_0, \aa_1, \xx_1)$, ..., ,
$(\xx_{\dd-1}, \aa_\dd, \xx_\dd)$. This implies $\aa = \aa_1 ... \aa_\dd$ in~$\MM$. The hypothesis that
each morphism $(\xx_{\ii-1}, \aa_\ii, \xx_\ii)$ is simple implies $\D_{\xx_{\ii-1}} \dive \aa_\ii \D_{\xx_\ii}$
for each~$\ii$, whence
$$\D_\xx \dive \aa_1 \D_{\xx_1} \dive \aa_1\aa_2 \D_{\xx_2}
\dive ... \dive \aa_1...\aa_\dd \D_{\xx_\dd} = \aa \D_{\xx \act \aa}.$$
Hence, for each element~$\aa$ in~$\MM_\xx$, there exists a unique element~$\aa'$ satisfying $\D_\xx \aa' =
\aa \D_{\xx\act\aa}$, and this is the element we define to be~$\f_\xx(\aa)$.

Then, $\xx \act \aa = \yy$ implies
$$(\xx, \aa, \yy)(\yy, \D_\yy, \f(\yy)) = (\xx, \D_\xx, \f(\xx))(\f(\xx), \f_\xx(\aa),
\f(\yy)).$$
By uniqueness, we deduce $\f((\xx, \aa, \yy)) = (\f(\xx), \f_\xx(\aa), \f(\yy))$.
\end{proof}

\subsection{Greatest common divisors}
\label{S:Gcd}

We observe---or rather recall---that left-gcd's always exist in a left-preGarside category. We begin with a
standard consequence of the noetherianity assumption~$\LG$.

\begin{lemm}
\label{L:Head}
Assume that $\CCC$ is a left-preGarside category  and $\SSS$ is a subset of $\Hom(\CCC)$ that contains
the identity-morphisms and is closed under right-lcm. Then every morphism has a unique
maximal left-divisor that lies in~$\SSS$.
\end{lemm}

\begin{proof}
Let $\ff$ be an arbitrary morphism. Starting from $\ff_0 = \id{\source\ff}$, which belongs to~$\SSS$ by
hypothesis, we construct a
$\div$-increasing sequence~$\ff_0, \ff_1, ...$ in~$\SSS \cap \Div(\ff)$. As long as $\ff_\ii$ is not
$\dive$-maximal in~$\SSS
\cap \Div(\ff)$, we can find~$\ff_{\ii+1}$ in~$\SS$ satisfying $\ff_\ii \div
\ff_{\ii+1} \dive \ff$. Condition~$\LGG$ implies that the construction stops after a finite number~$\dd$ of
steps. Then $\ff_\dd$ is a maximal left-divisor of~$\ff$ lying in~$\SSS$. 

As for uniqueness, assume that $\gg'$ and $\gg''$ are maximal left-divisors of~$\ff$ that lie in~$\SSS$. By
construction, $\gg'$ and $\gg''$ admit a common right-multiple, namely~$\ff$, hence, by~$\LGGii$, they
admit a right-lcm~$\gg$. By construction, $\gg$ is a left-divisor of~$\ff$, and it belongs to~$\SSS$ since $\gg'$
and $\gg''$ do. The maximality of~$\gg$ and~$\gg'$ implies $\gg' = \gg = \gg''$.
\end{proof}

\begin{prop}
\label{P:Gcd}
Assume that $\CCC$ is a left-preGarside category. Then any two morphisms of~$\CCC$ sharing a common
source admit a unique left-gcd.
\end{prop}

\begin{proof}
Let $\SSS$ be the family of all common left-divisors of~$\ff$ and~$\gg$. It
contains~$\id{\source\ff}$, and it is closed under~$\lcm$. A left-gcd of~$\ff$ and~$\gg$
is a maximal left-divisor of~$\ff$ lying in~$\SSS$. Lemma~\ref{L:Head} gives the result.
\end{proof}

\subsection{Least common multiples}

As for right-lcm, the axioms of left-Garside categories only demand that a right-lcm exists when a common
right-multiple does. A necessary condition for such a common right-multiple to exist is to share a common
source. This condition is also sufficient. Again we begin with an auxiliary result.

\begin{lemm}
\label{L:Lcm}
Assume that $\CCC$ is a left-Garside category. Then, for $\ff = \ff_1 ... \ff_\dd$ with $\ff_1, ..., \ff_\dd$
simple and
$\xx = \source{\ff}$, we have 
\begin{equation}
\ff \dive \D(\xx) \, \D(\f(\xx))\,  ... \, \D(\f^{\dd-1}(\xx)).
\end{equation}
\end{lemm}

\begin{proof}
We use induction on~$\dd$. For $\dd = 1$, this is the definition of simplicity. Assume $\dd \ge 2$. Put
$\yy = \target{\ff_1}$. Applying the induction hypothesis to $\ff_2 ... \ff_\dd$, we
find
\begin{align*}
\ff = \ff_1 (\ff_2 ... \ff_\dd)
&\dive \ff_1 \, \D(\yy) \, \D(\f(\yy)) \, ...\, \D(\f^{\dd-2}(\yy))\\
&= \D(\xx) \, \D(\f(\xx)) \, ...\, \D(\f^{\dd-2}(\xx)) \, \f^{\dd-1}(\ff_1)\\
&\dive \D(\xx) \, \D(\f(\xx)) \, ...\, \D(\f^{\dd-2}(\xx)) \, \D(\f^{\dd-1}(\xx)).
\end{align*}
The second equality comes from applying~\eqref{E:Phi} $\dd-1$ times, and the last inequality comes from
the fact that $\f^{\dd-1}(\ff_1)$ is simple with source~$\f^{\dd-1}(\xx)$.
\end{proof}

\begin{prop}
\label{P:Lcm}
Assume that $\CCC$ is a left-Garside category. Then any two morphisms of~$\CCC$ sharing a common
source admit a unique right-lcm.
\end{prop}

\begin{proof}
Let $\ff, \gg$ be any two morphisms with source~$\xx$. By
Lem\-ma~\ref{L:Basic}, there exists~$\dd$ such that $\ff$ and $\gg$ can be expressed as the product of at
most $\dd$ simple morphisms. Then, by Lemma~\ref{L:Lcm}, $\D(\xx) \, \D(\f(\xx))\,  ... \,
\D(\f^{\dd-1}(\xx))$ is a common right-multiple of~$\ff$ and~$\gg$. Finally, $\LGGii$ implies that $\ff$ and
$\gg$ admit a right-lcm. The uniqueness of the latter is guaranteed by Lem\-ma~\ref{L:Basic}$(i)$.
\end{proof}

In a general context of categories, right-lcm's are usually called \emph{push-outs} (wher\-eas
left-lcm's are called \emph{pull-backs}). So Proposition~\ref{P:Lcm} states that every left-Garside category
admits pushouts.

Applying the previous results to the special case of categories associated with a partial action gives analogous
results for all locally left-Garside monoids.

\begin{coro}
\label{C:Lcm}
Assume that $\MM$ is a locally left-Garside monoid with respect to some partial action of~$\MM$ on~$\XX$.

$(i)$ Any two elements of~$\MM$ admit  a unique left-gcd and a unique right-lcm. 

$(ii)$ For each~$\xx$ in~$\XX$, the
subset~$\MM_\xx$ of~$\MM$ is closed under right-lcm.
\end{coro}

\begin{proof}
$(i)$ As for left-gcd's, the result directly follows from Proposition~\ref{P:Gcd} since, by definition, $\MM$ is
left-preGarside.

As for right-lcm's, assume that $\MM$ is locally
left-Garside with left-Garside sequence $(\D_\xx)_{\xx\in\XX}$. Let $\aa, \bb$ be two elements of~$\MM$. By
definition of a partial action, there exists~$\xx$ in~$\XX$ such that both $\xx \act\aa$ and $\xx\act\bb$ are
defined. By Proposition~\ref{P:Lcm}, $(\xx, \aa, \xx\act\aa)$ and $(\xx, \bb, \xx\act\bb)$ admit a 
right-lcm $(\xx, \cc, \zz)$ in the category~$\CMX$. By construction, $\cc$ is a common right-multiple of~$\aa$
and~$\bb$ in~$\MM$. As $\MM$ is assumed to satisfy~$\LGii$, $\aa$ and $\bb$ admit a right-lcm
in~$\MM$. 

$(ii)$ Fix now~$\xx$ in~$\XX$, and let $\aa, \bb$ belong to~$\MM_\xx$, \ie, assume that $\xx\act\aa$ and
$\xx\act\bb$ are defined. Then $(\xx, \aa, \xx\act\aa)$ and $(\xx, \bb, \xx\act\bb)$ are morphisms
of~$\CMX$. As above, they admit a right-lcm, which must be $(\xx, \cc, \xx\act \cc)$ when $\cc$ is the
right-lcm of~$\aa$ and~$\bb$. Hence $\cc$ belongs to~$\MM_\xx$.
\end{proof}

\section{Regular left-Garside categories}
\label{S:Normal}

The main interest of Garside structures is the existence of a canonical normal forms, the so-called greedy
normal form~\cite{Eps}. In this section, we adapt the construction of the normal form to the context of
left-Garside categories---this was done in~\cite{DiM} already---and of locally left-Garside monoids.
The point here is that studying the computation of the normal form naturally leads to introducing
the notion of a regular left-Garside category, crucial in Section~\ref{S:Main}.

\subsection{The head of a morphism}

By Lemma~\ref{L:Basic}$(ii)$, every morphism in a left-Garside category is a product
of simple morphisms. The decomposition need not to be unique in general, and the first step for constructing
a normal form consists in isolating a particular simple morphism that left-divides the considered morphism.
It will be useful to develop the construction in a general framework where the distinguished morphisms need
not necessarily be the simple ones. 

\begin{nota}
We recall that, for $\ff, \gg$ in~$\Hom(\CCC)$, where $\CCC$ is a left-preGarside category, 
$\lcm(\ff, \gg)$ is the right-lcm of~$\ff$ and~$\gg$, when it exists. In this case, we denote
by~$\ff\under\gg$ the unique morphism that satisfies
\begin{equation}
\ff \cdot \ff\under\gg = \lcm(\ff,\gg).
\end{equation}
\end{nota}

We use a similar notation in the case of a (locally) left-Garside monoid. 

\begin{defi}
\label{D:Seed}
Assume that $\CCC$ is a left-preGarside category and $\SSS$ is included in~$\Hom(\CCC)$. We say
that $\SSS$ is a \emph{seed} for~$\CCC$ if 

$(i)$ $\SSS$ left-generates~$\Hom(\CCC)$,

$(ii)$ $\SSS$ is closed under the operations~$\lcm$ and~$\under$,

$(iii)$ $\SSS$ is closed under left-divisor.
\end{defi}

In other words, $\SSS$ is a seed for~$\CCC$ if $(i)$ every nontrivial morphism of~$\CCC$ is left-divisible by a
nontrivial element of~$\SSS$, $(ii)$ for all~$\ff,\gg$ in~$\SSS$, the morphisms
$\lcm(\ff,\gg)$ and $\ff\under\gg$ belong to~$\SSS$ whenever they exist, and $(iii)$ for each~$\ff$
in~$\SSS$, the relation $\hh \dive \ff$ implies $\hh \in \SSS$.

\begin{lemm}
If $\CCC$ is a left-Garside category, then $\Sim(\CCC)$ is a seed for~$\CCC$.
\end{lemm}

\begin{proof}
First, $\Sim(\CCC)$ left-generates~$\Hom(\CCC)$ by Condition~$\LGGiii$.

Next, assume that $\ff, \gg$ are simple morphisms sharing the same source~$\xx$. By
Proposition~\ref{P:Lcm}, the morphisms $\lcm(\ff,\gg)$ and $\ff\under\gg$ exist.
By definition, we have $\ff \dive \D(\xx)$ and $\gg
\dive \D(\xx)$, hence $\lcm(\ff, \gg) \dive \D_\xx$. Hence $\lcm(\ff,\gg)$ is simple. Let $\hh$ satisfy
$\lcm(\ff,\gg)\hh = \D(\xx)$. This is also $\ff \,(\ff\under\gg)\,\hh = \D(\xx)$. By
Lemma~\ref{L:Simple}$(ii)$,
$(\ff\under\gg)\,\hh$, which is a right-divisor of~$\D(\xx)$, is simple, and, therefore, $\ff\under\gg$, which
is a left-divisor of $(\ff\under\gg)\,\hh$, is simple as well by transitivity of~$\dive$.

Finally, $\Sim(\CCC)$ is closed under left-divisor by definition.
\end{proof}

Lemma~\ref{L:Head} guarantees that, if $\SSS$ is a seed for~$\CCC$, then every morphism~$\ff$
of~$\CCC$ has a unique maximal left-divisor~$\gg$ lying in~$\SSS$, and Condition~$(i)$ of
Definition~\ref{D:Seed} implies that $\gg$ is nontrivial whenever $\ff$ is.

\begin{defi}
In the context above, the morphism~$\gg$ is called the
\emph{$\SSS$-head} of~$\ff$, denoted~$\Head\SSS\ff$.
\end{defi}

In the case of~$\Sim(\CCC)$, it is easy to check, for each~$\ff$ in~$\Hom(\CCC)$, the equality
\begin{equation}
\Head{\Sim(\CCC)}\ff = \gcd(\ff, \D(\source\ff));
\end{equation}
in this case, we shall simply write~$\head\ff$ for $\Head{\Sim(\CCC)}\ff$.

\subsection{Normal form}

The following result is an adaptation of a result that is classical in the framework
of Garside monoids.

\begin{prop}
\label{P:NF}
Assume that $\CCC$ is a left-preGarside category and $\SSS$ is a seed for~$\CCC$. 
Then every nontrivial morphism~$\ff$ of~$\CCC$ admits a unique decomposition
\begin{equation}
\label{E:Normal}
\ff = \ff_1 ...\ff_\dd,
\end{equation}
where $\ff_1, ..., \ff_\dd$ lie in~$\SSS$, $\ff_\dd$ is nontrivial, and
$\ff_\ii$ is the $\SSS$-head of~$\ff_\ii ... \ff_\dd$ for each~$\ii$.
\end{prop}

\begin{proof}
Let $\ff$ be a nontrivial morphism of~$\CCC$, and let $\ff_1$ be the $\SSS$-head of~$\ff$. Then $\ff_1$
belongs to~$\SSS$, it is nontrivial, and we have $\ff = \ff_1 \ff'$ for some unique~$\ff'$.
If $\ff'$ is trivial, we are done, otherwise we repeat the argument with~$\ff'$. In this way we obtain a
$\div$-increasing sequence $\id{\source\ff} \div \ff_1 \div \ff_1\ff_2 \div ...$\,. Condition~$\LGG$ implies
that the construction stops after a finite number of steps, yielding a decomposition of the
form~\eqref{E:Normal}.

As for uniqueness, assume that $(\ff_1, ..., \ff_\dd)$ and $(\gg_1, ..., \gg_\ee)$ are decomposition of~$\ff$
that satisfy the conditions of the statement. We prove $(\ff_1, ..., \ff_\dd) = (\gg_1, ..., \gg_\ee)$ using
induction on $\min(\dd, \ee)$. First, $\dd = 0$ implies $\ee = 0$ by Lemma~\ref{L:Basic}$(i)$. Otherwise, the
hypotheses imply $\ff_1 = \Head{\SSS}{\ff} = \gg_1$. Left-cancelling~$\ff_1$ gives two decompositions
$(\ff_2, ..., \ff_\dd)$ and $(\gg_2, ..., \gg_\ee)$ of~$\ff_2...\ff_\dd$, and we apply the
induction hypothesis.
\end{proof}

\begin{defi}
In the context above, the sequence $(\ff_1, ..., \ff_\dd)$ is called the \emph{$\SSS$-normal form} of~$\ff$.
\end{defi}

When $\SSS$ turns out to be the family~$\Sim(\CCC)$, the $\SSS$-normal form will be simply called the
\emph{normal form}. The interest of the $\SSS$-normal form lies in that
it is easily characterized and easily computed. First, one has the following local characterization of normal
sequences.

\begin{prop}
\label{P:Local}
Assume that $\CCC$ is a left-preGarside category and $\SSS$ is a seed for~$\CCC$. 
Then a sequence of morphisms $(\ff_1, ..., \ff_\dd)$ is  $\SSS$-normal if and
only if each length two subsequence $(\ff_\ii, \ff_{\ii+1})$ is  $\SSS$-normal.
\end{prop}

This follows from an auxiliary lemma.

\begin{lemm}
\label{L:Aux}
Assume that $(\ff_1, \ff_2)$ is  $\SSS$-normal and $\gg$ belongs to~$\SSS$. Then $\gg \dive \ff\ff_1\ff_2$
implies $\gg \dive \ff\ff_1$.
\end{lemm}

\begin{proof}
The hypothesis implies that $\ff$ and $\gg$ have the same source. Put $\gg' = \ff\under\gg$. The
hypothesis that $\SSS$ is closed under~$\under$ and an easy induction on the length of the $\SSS$-normal
form of~$\ff$ show that $\gg'$ belongs to~$\SSS$. By hypothesis, we have both $\gg \dive \ff\ff_1\ff_2$ and
$\ff \dive \ff\ff_1\ff_2$, hence $\lcm(\ff,\gg) = \ff \gg' \dive \ff\ff_1\ff_2$ whence $\gg' \dive \ff_1\ff_2$
by left-cancelling~$\ff$. As $\gg$ belongs to~$\SSS$ and $(\ff_1, \ff_2)$ is normal, this implies $\gg' \dive
\ff_1$, and finally $\gg \dive \ff\gg' \dive \ff\ff_1$.
\end{proof}

\begin{proof}[Proof of Proposition~\ref{P:Local}]
It is enough to consider the case $\dd = 2$, from which an easy induction on~$\dd$ gives the general case.
So we assume that $(\ff_1, \ff_2)$ and $(\ff_2, \ff_3)$ are $\SSS$-normal, and aim at proving that $(\ff_1,
\ff_2,
\ff_3)$ is $\SSS$-normal. The point is to prove that, if $\gg$ belongs to~$\SSS$, then $\gg \dive \ff_1\ff_2\ff_3$ implies $\gg
\dive \ff_1$. So assume $\gg \dive\ff_1\ff_2\ff_3$. As $(\ff_2, \ff_3)$ is $\SSS$-normal, Lemma~\ref{L:Aux}
implies $\gg \dive \ff_1\ff_2$. As $(\ff_1, \ff_2)$ is $\SSS$-normal, this implies $\gg \dive \ff_1$.
\end{proof}

\subsection{A computation rule}

We establish now a recipe for inductively computing the $\SSS$-normal form, namely determining the
$\SSS$-normal form of~$\gg\ff$ when that of~$\ff$ is known and $\gg$ belongs to~$\SSS$. 

\begin{prop}
\label{P:LeftProduct}
Assume that $\CCC$ is a left-preGarside category, $\SSS$ is a seed for~$\CCC$, and
$(\ff_1, ..., \ff_\dd)$ is the $\SSS$-normal form of~$\ff$. Then, for each~$\gg$ in~$\SSS$, the
$\SSS$-normal form of~$\gg\ff$ is $(\ff'_1, ..., \ff'_\dd, \gg_\dd)$, where
$\gg_0 = \gg$ and $(\ff'_\ii, \gg_\ii)$ is the $\SSS$-normal form of~$\gg_{\ii-1}
\ff_\ii$ for~$\ii$ increasing from~$1$ to~$\dd$---see Figure~\ref{F:LeftProduct}.
\end{prop}

\begin{figure}[tb]
\begin{picture}(45,15)(0,0)
\put(0,0){\includegraphics{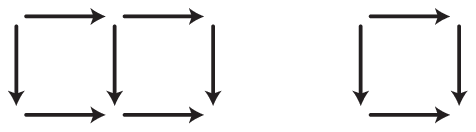}}
\put(4,-2.5){$\ff_1$}
\put(14,-2.5){$\ff_2$}
\put(4,12.5){$\ff'_1$}
\put(14,12.5){$\ff'_2$}
\put(-3,5){$\gg_0$}
\put(12,5){$\gg_1$}
\put(22,5){$\gg_2$}
\put(27,1){...}
\put(27,11){...}
\put(36.5,5){$\gg_{\dd\!-\!1}$}
\put(47,5){$\gg_\dd$}
\put(39,-2.5){$\ff_\dd$}
\put(39,12.5){$\ff'_\dd$}
\end{picture}
\caption{\sf\smaller Adding one $\SSS$-factor~$\gg_0$ on the left of an $\SSS$-normal sequence $(\ff_1, ...,
\ff_\dd)$: compute the $\SSS$-normal form~$(\ff'_1, \gg_1)$ of~$\gg_0\ff_1$, then the $\SSS$-normal
form~$(\ff'_2, \gg_2)$ of~$\gg_1\ff_2$, and so on from left to right; the sequence $(\ff'_1, ..., \ff'_\dd,
\gg_\dd)$ is $\SSS$-normal.}
\label{F:LeftProduct}
\end{figure}

\begin{proof}
For an induction, it is enough to consider the case $\dd = 2$, hence to prove

\begin{claim}
Assume that the diagram
\vrule width0pt height9mm depth9mm
\begin{picture}(29,6)(-4,5)
\put(0,0){\includegraphics{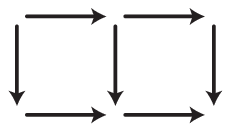}}
\put(4,-2.5){$\ff_1$}
\put(14,-2.5){$\ff_2$}
\put(4,12.5){$\ff'_1$}
\put(14,12.5){$\ff'_2$}
\put(-3,6){$\gg_0$}
\put(12,6){$\gg_1$}
\put(22,6){$\gg_2$}
\end{picture}
is commutative and $(\ff_1, \ff_2)$ and $(\ff'_1, \gg_1)$ are $\SSS$-normal. Then $(\ff'_1, \ff'_2)$ is $\SSS$-normal.
\end{claim}

\noindent So assume that $\hh$ belongs to~$\SSS$ and satisfies $\hh \dive \ff'_1 \ff'_2$. Then, a fortiori,  we
have $\hh \dive \ff'_1 \ff'_2 \gg_2 = \gg_0 \ff_1 \ff_2$, hence $\hh \dive \gg_0\ff_1$ by Lemma~\ref{L:Aux}
since $(\ff_1, \ff_2)$ is $\SSS$-normal. Therefore we have $\hh \dive \ff'_1 \gg_1$, hence $\hh \dive \ff'_1$
since $(\ff'_1, \gg_1)$ is $\SSS$-normal.
\end{proof}

The results of Proposition~\ref{P:Local} and~\ref{P:LeftProduct} apply in particular when $\CCC$ is
left-Garside and $\SSS$ is the family of all simple morphisms, in which case they involve the standard
normal form.

In the case of lcm's, Corollary~\ref{C:Lcm} shows how a result established for
general left-Garside categories can induce a similar result for locally left-Garside monoids. The situation is
similar with the normal form, provided some additional assumption is satisfied.

\begin{defi}
\label{D:Coherent}
A left-Garside sequence $(\D_\xx)_{\xx \in \XX}$ witnessing that a certain monoid is locally
left-Garside is said to be \emph{coherent} if, for all~$\aa, \xx, \xx'$ such that $\aa \act \xx$ is defined,
$\aa\dive\D_{\xx'}$ implies $\aa \dive \D_\xx$. 
\end{defi}

For instance, the family $(\D_\nn)_{\nn\in\Nat}$ witnessing for the locally left-Garside structure of the
monoid~$\BP\infty$ is coherent. Indeed, a positive $\nn$-strand braid~$\aa$ is a left-divisor of~$\D_\nn$ if
and only if it is a left-divisor of~$\D_{\nn'}$ for every $\nn' \ge \nn$. The reason is that being simple is an
intrinsic property of positive braids: a positive braid is simple if and only if it can be represented by a braid
diagram in which any two strands cross at most once~\cite{ElM}.

\begin{prop}
\label{P:Coherent}
Assume that $\MM$ is a locally left-Garside monoid associated with a coherent left-Garside sequence
$(\D_\xx)_{\xx \in \XX}$. Let $\Sig = \{\aa\in\MM \mid \exists\xx\in\XX(\aa\dive \D_\xx)\}$. Then $\Sig$ is
a seed for~$\MM$, every element of~$\MM$ admits a unique $\Sig$-normal form, and the counterpart of
Propositions~\ref{P:Local} and~\ref{P:LeftProduct} hold for the $\Sig$-normal form in~$\MM$.
\end{prop}

\begin{proof}
Axiom~$\LGloc$ guarantees that every nontrivial element of~$\MM$ is left-divisible by some nontrivial
element of~$\Sig$. Then, by hypothesis, $\id\xx \dive \D_\xx$ holds for each object~$\xx$. Then, assume
$\aa, \bb \in \Sig$. There exists~$\xx$ such that $\xx\act\aa$ and $\xx\act\bb$ is defined. By definition
of~$\Sig$, there exists~$\xx'$ satisfying $\aa \dive \D_{\xx'}$, hence, by definition of coherence, we have
$\aa \dive \D_\xx$. A similar argument gives $\bb \dive\D_\xx$, whence $\lcm(\aa,\bb) \dive \D_\xx$, and
$\lcm(\aa,\bb) \in
\Sig$. So there exists~$\cc$ satisfying $\aa \, (\aa\under\bb) \, \cc = \D_\xx$. By~$\LGloc$, we deduce
$(\aa\under\bb)\,\cc \dive \D_{\xx\act\aa}$, whence $\aa\under\bb \dive \D_{\xx\act\aa}$, and we
conclude that $\aa\under\bb$ belongs to~$\Sig$. Finally, it directly results from its definition that $\Sig$ is
closed under left-divisor. Hence $\Sig$ is a seed for~$\MM$ in the sense of Definition~\ref{D:Seed}.

As, by definition, $\MM$ is a left-preGarside monoid, Proposition~\ref{P:NF} applies, guaranteeing the
existence and uniqueness of the $\Sig$-normal form on~$\MM$, and so do Propositions~\ref{P:Local}
and~\ref{P:LeftProduct}.
\end{proof}

Thus, the good properties of the greedy normal form are preserved when the assumption that a
global Garside element~$\D$ exists is replaced by the weaker assumption that local Garside
elements~$\D_\xx$ exist, provided they satisfy some coherence. 

\subsection{Regular left-Garside categories}

It is natural to look for a counterpart of the recipe of Proposition~\ref{P:LeftProduct} involving
right-multiplication by an element of the seed instead of left-multiplication. Such a counterpart exists but,
interestingly, the situation is not symmetric, and we need a new argument. The latter demands
that the considered category satisfies an additional condition, which is automatically satisfied in  a two-sided
Garside category, but not in a left-Garside category.

In this section, we only consider the case of a left-Garside category and its simple morphisms, and not the case
of a general left-preGarside category with an arbitrary seed---see Remark~\ref{R:Seed}. So, we only refer to the
standard normal form.

\begin{defi}
We say that a left-Garside category~$\CCC$ is \emph{regular} if the functor~$\f$ preserves normality of
length~$2$ sequences: for $\ff_1, \ff_2$ simple with $\target{\ff_1} = \source{\ff_2}$,
\begin{equation}
(\ff_1, \ff_2)
\mbox{\ normal \quad implies \quad}
(\f(\ff_1), \f(\ff_2))
\mbox{\ normal.}
\end{equation}
\end{defi}

\begin{prop}
\label{P:RightProduct}
Assume that $\CCC$ is a regular left-Garside category, that $(\ff_1, ..., \ff_\dd)$ is
the normal form of a morphism~$\ff$, and that $\gg$ is simple.  
Then the normal form of~$\ff\gg$ is $(\gg_0, \ff'_1, ..., \ff'_\dd)$, where
$\gg_\dd = \gg$ and $(\gg_{\ii-1}, \ff'_\ii)$ is the normal form
of~$\ff_\ii \gg_\ii$ for $\ii$ decreasing from~$\dd$ to~$1$---see
Figure~\ref{F:RightProduct}.
\end{prop}

\begin{figure}[tb]
\begin{picture}(45,15)(0,0)
\put(0,0){\includegraphics{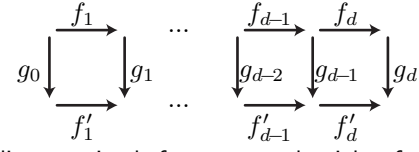}}
\put(4,-2.5){$\ff'_1$}
\put(4,12.5){$\ff_1$}
\put(-3,5){$\gg_0$}
\put(12,5){$\gg_1$}
\put(17,1){...}
\put(17,11){...}
\put(26.5,5){$\gg_{\dd\!-\!2}$}
\put(36.5,5){$\gg_{\dd\!-\!1}$}
\put(47,5){$\gg_\dd$}
\put(27.5,-2.5){$\ff'_{\dd\!-\!1}$}
\put(27.5,12.5){$\ff_{\dd\!-\!1}$}
\put(39,-2.5){$\ff'_\dd$}
\put(39,12.5){$\ff_\dd$}
\end{picture}
\caption{\sf\smaller Adding one simple factor~$\gg_\dd$ on the right of a simple sequence $(\ff_1, ...,
\ff_\dd)$: compute the normal form~$(\gg_{\dd-1} \ff'_\dd)$ of~$\ff_\dd \gg_\dd)$, then the normal
form~$(\gg_{\dd-2} \ff'_{\dd-2})$ of~$\ff_{\dd-1} \gg_{\dd-1}$, and so on from right to left; the sequence
$(\gg_0, \ff'_1, ..., \ff'_\dd)$ is normal.}
\label{F:RightProduct}
\end{figure}

We begin with an auxiliary observation.

\begin{lemm}
\label{L:Coprime}
Assume that $\CCC$ is a left-Garside category and $\ff_1, \ff_2$ are
simple morphisms satisfying $\target\ff_1 = \source\ff_2$. Then  $(\ff_1, \ff_2)$ is normal if and
only if $\ff_1^*$ and
$\ff_2$ are left-coprime, \ie, $\gcd(\ff_1^*, \ff_2)$ is trivial.
\end{lemm}

\begin{proof}
The following equalities always hold:
$$\head{\ff_1\ff_2} 
= \gcd(\ff_1\ff_2, \D(\source{\ff_1}))
= \gcd(\ff_1\ff_2, \ff_1 \ff_1^*)
= \ff_1\, \gcd(\ff_2, \ff_1^*).$$
Hence $(\ff_1, \ff_2)$ is normal, \ie, $\ff_1 = \head{\ff_1\ff_2}$ holds, if and only if
$\ff_1 = \ff_1\, \gcd(\ff_2, \ff_1^*)$ does, which is $\gcd(\ff_2, \ff_1^*) =
\id{\target{\ff_1}}$ as left-cancelling~$\ff_1$ is allowed.
\end{proof}

\begin{proof}[Proof of Proposition~\ref{P:RightProduct}]
As in the case of Proposition~\ref{P:LeftProduct} it is enough to consider the case $\dd = 2$, and therefore it is
enough to prove

\begin{claim}
Assume that the diagram
\vrule width0pt height9mm depth9mm
\begin{picture}(29,6)(-4,5)
\put(0,0){\includegraphics{LeftProduct.eps}}
\put(4,-2.5){$\ff'_1$}
\put(14,-2.5){$\ff'_2$}
\put(4,12.5){$\ff_1$}
\put(14,12.5){$\ff_2$}
\put(-3,6){$\gg_0$}
\put(12,6){$\gg_1$}
\put(22,6){$\gg_2$}
\end{picture}
is commutative and $(\ff_1, \ff_2)$ and $(\gg_1, \ff_2)$ are normal. Then $(\ff'_1, \ff'_2)$ is normal.
\end{claim}

\noindent To prove the claim, we introduce the morphisms~$\gg_0^*, \gg_1^*, \gg_2^*$ defined
by $\gg_\ii \gg_\ii^* = \D(\source{\gg_\ii})$ (Definition~\ref{D:Simple}). Then the diagram
\vrule width0pt height8mm depth11mm
\begin{picture}(31,16)(-5,0)
\put(0,-10){\includegraphics{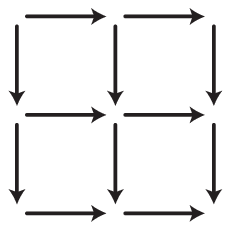}}
\put(4,2.5){$\ff'_1$}
\put(14,2.5){$\ff'_2$}
\put(4,12.5){$\ff_1$}
\put(14,12.5){$\ff_2$}
\put(-3,5.5){$\gg_0$}
\put(12,5.5){$\gg_1$}
\put(22,5.5){$\gg_2$}
\put(-3,-5){$\gg_0^*$}
\put(12,-5.5){$\gg_1^*$}
\put(22,-5){$\gg_2^*$}
\put(2,-13){$\f(\ff_1)$}
\put(12,-13){$\f(\ff_2)$}
\end{picture}
 is commutative. Indeed, appying~\eqref{E:Phi}, we find
\begin{align*}
\gg_0 \ff'_1 \gg_1^* = \ff_1 \gg_1 \gg_1^* 
&= \ff_1 \D(\source{\gg_0}) = \ff_1
\D(\target{\ff_1})\\
&= \D(\source{\ff_1}) \f(\ff_1)
= \D(\source{\gg_0}) \f(\ff_1)
= \gg_0 \gg_0^* \f(\ff_1),
\end{align*}
hence $\ff'_1\gg_1^* = \gg_0^* \f(\ff_1)$ by left-cancelling~$\gg_0$. A  similar argument gives $\ff'_2
\gg_2^* =\nobreak \gg_1^* \f(\ff_2)$.

Assume that $\hh$ is simple and satisfies $\hh \dive \ff'_1 \ff'_2$. We deduce 
$$\hh \dive \ff'_1
\ff'_2\gg_2^* = \gg_0^* \f(\ff_1) \f(\ff_2).$$
By hypothesis, $(\ff_1, \ff_2)$ is normal. Hence the hypothesis that $\CCC$ is regular implies that $(\f(\ff_1),
\f(\ff_2))$ is normal as well. By Lemma~\ref{L:Aux}, $\hh \dive \gg_0^*
\f(\ff_1) \f(\ff_2)$ implies $\hh \dive\nobreak \gg_0^* \f(\ff_1) = \ff'_1 \gg_1^*$. We deduce $\hh \dive
\gcd(\ff'_1\ff'_2, \ff'_1\gg_1^*) = \ff'_1 \, \gcd(\ff'_2, \gg_1^*)$. By Lemma~\ref{L:Coprime}, the hypothesis
that $(\gg_1, \ff'_1)$ is normal implies $\gcd(\gg_1^*, \ff'_2) = 1$, and, finally, we deduce $\hh \dive
\ff'_1$, \ie, $(\ff'_1, \ff'_2)$ is normal.
\end{proof}

\begin{rema}
\label{R:Seed}
It might be tempting to mimick the arguments of this section in the general framework of a left-preGarside
category~$\CCC$ and a seed~$\SSS$, provided some additional conditions are satisfied. However, it is unclear
that the extension can be a genuine one. For instance, if we require that, for
each~$\ff$ in~$\SSS$, there exists~$\ff^*$ in~$\SSS$ such that $\ff\ff^*$ exists and depends on~$\source\ff$
only, then the map $\source\ff \mapsto \ff\ff^*$ is a left-Garside map and we are back to left-Garside
categories.
\end{rema}

\subsection{Regularity criteria}
\label{S:Criteria}

We conclude with some sufficient conditions implying regularity. In particular, we observe that, in the
two-sided case, regularity is automatically satisfied.

\begin{lemm}
\label{L:Bijective}
Assume that $\CCC$ is a left-Garside category~$\CCC$. Then a sufficient condition for~$\CCC$ to be regular is
that the functor~$\f$ is bijective on~$\Hom(\CCC)$.
\end{lemm}

\begin{proof}
Assume that $\CCC$ is a left-Garside category and $\f$ is bijective on~$\Hom(\CCC)$.
First we claim that $\f(\ff) \dive \f(\gg)$ implies $\ff \dive \gg$ in~$\CCC$. Indeed, assume $\f(\gg) = \f(\ff)
\hh$. As $\f$ is surjective, we have $\f(\gg) = \f(\ff) \f(\hh')$ for some~$\hh'$, hence $\f(\gg) = \f(\ff\hh')$
since $\f$ is a functor, hence $\gg = \ff \hh'$ since $\f$ is injective.

Next, we claim that, if $\f(\ff)$ is simple if and only if $\ff$ is simple. That the condition is sufficient directly
follows from Definition~\ref{D:Simple}. Conversely, assume that $\f(\ff)$ is simple. This means that there
exists $\gg$ satisfying $\f(\ff) \gg = \D(\source{\f(\ff)})$. As $\f$ is surjective, there exists $\gg'$ satisfying
$\gg = \f(\gg')$. Applying~\eqref{E:Dual}, we obtain $\f(\ff\gg') = \D(\source{\f(\ff)}) = \f(\D(\source{\ff})$,
hence $\ff\gg' = \D(\source{\ff})$ by injectivity of~$\f$.

Finally, assume that $(\ff_1, \ff_2)$ is normal, and $\gg$ is a simple morphism left-dividing $\f(\ff_1)
\f(\ff_2)$, hence satisfying $\gg\hh = \f(\ff_1) \f(\ff_2)$ for some~$\hh$. As $\f$ is surjective, we have
$\gg = \f(\gg')$ and $\hh = \f(\hh')$ for some~$\gg', \hh'$. Moreover, by the claim above, the hypothesis that
$\gg$ is simple implies that $\gg'$ is simple as well. Then we have $\f(\gg') \f(\hh') = \f(\ff_1)
\f(\ff_2)$, hence $\gg'\hh' = \ff_1 \ff_2$ since $\f$ is a functor and it is injective. The hypothesis that
$(\ff_1, \ff_2)$ is normal implies $\gg' \dive \ff_1$, hence $\gg = \f(\gg') \dive \f(\ff_1)$. So $(\f(\ff_1),
\f(\ff_2))$ is normal, and $\CCC$ is regular.
\end{proof}

\begin{prop}
Every Garside category is regular.
\end{prop}

\begin{proof}
Assume that $\CCC$ is a left-Garside with respect to~$\D$ and right-Garside with respect to~$\nabla$
satisfying $\D(\xx) = \nabla(\xx')$ for $\xx' = \target{\D(\xx)}$. Put
$\psi(\xx') = \source{\nabla(\xx')}$ for $\xx'$ in~$\Obj(\CCC)$ and, for~$\gg$ simple
in~$\Hom(\CCC)$, hence a right-divisor of~$\nabla(\target{\gg})$, denote by~${}^*\gg$ the unique simple
morphism satisfying ${}^*\gg \gg = \nabla(\source{\gg})$, and put $\psi(\gg) = {}^{**}\gg$. Then arguments
similar to those of Lemma~\ref{L:Basic} give the equality
\begin{equation}
\label{E:Psi}
\nabla(\source\gg) \, \gg = \psi(\gg) \, \nabla(\target\gg) 
\end{equation}
which is an exact counterpart of~\eqref{E:Phi}. Let $\ff : \xx \to \yy$ be any morphism in~$\CCC$. Put $\xx'
= \f(\xx)$ and $\yy' = \f(\yy)$. By construction, we also have $\xx = \psi(\xx')$ and $\yy = \psi(\yy')$.
Applying~\eqref{E:Phi} to~$\ff : \xx \to \yy$, we obtain $\D(\xx) \, \f(\ff) = \ff \, \D(\yy),$
which is also
\begin{equation}
\label{E:DN1}
\nabla(\xx') \, \f(\ff) = \ff \, \nabla(\yy').
\end{equation}
On the other hand, applying~\eqref{E:Psi} to~$\f(\ff) : \xx' \to \yy'$ yields
\begin{equation}
\label{E:DN2}
\nabla(\xx') \, \f(\ff) = \psi(\f(\ff))\, \nabla(\yy').
\end{equation}
Comparing~\eqref{E:DN1} and~\eqref{E:DN2} and right-cancelling~$\nabla(\yy')$, we deduce $\psi(\f(\ff)) =
\ff$. A symmetric argument gives $\f(\psi(\gg)) = \gg$ for each~$\gg$, and we conclude that $\psi$ is the
inverse of~$\f$, which is therefore bijective. Then we apply Lemma~\ref{L:Bijective}.
\end{proof}

\begin{rema}
The above proof shows that, if $\CCC$ is a left-Garside category that is Garside, then the associated
functor~$\f$ is bijective both on~$\Obj(\CCC)$ and on~$\Hom(\CCC)$. Let us mention without proof that
this necessary condition is actually also sufficient.
\end{rema}

Apart from the previous very special case, we can state several weaker regularity criteria that are close to
the definition and will be useful in Section~\ref{S:Main}. We recall that
$\head\ff$ denoted the maximal simple morphism left-dividing~$\ff$.

\begin{prop}
\label{P:CriterionBis}
A left-Garside category $\CCC$ is regular if and only if $\f$ preserves the head function on
product of two simples: for $\ff_1,
\ff_2$ simple with $\target{\ff_1} = \source{\ff_2}$,
\begin{equation}
\label{E:PhiHead}
\head{\f(\ff_1 \ff_2)} = \f(\head{\ff_1 \ff_2});
\end{equation}
\end{prop}

\begin{proof}
Assume that $\CCC$ is regular, and let $\ff_1, \ff_2$ satisfy $\target\ff_1 = \source\ff_2$. Let $(\ff'_1,
\ff'_2)$ be the formal form of~$\ff_1\ff_2$---which has length~$2$ at most by Proposition~\ref{P:LeftProduct}.
Then, $(\f(\ff'_1), \f(\ff'_2))$ is normal and satisfies $\f(\ff'_1)\f(\ff'_2) = \f(\ff_1\ff_2)$, so $(\f(\ff'_1),
\f(\ff'_2))$ is the normal form of~$f(\ff_1\ff_2)$. Hence we have
$\head{\ff_1\ff_2} = \ff'_1$ and $\head{\f(\ff_1\ff_2)} = \f(\ff'_1)$, which is~\eqref{E:PhiHead}.

Conversely, assume~\eqref{E:PhiHead} and let $(\ff_1, \ff_2)$ be normal. By construction, we have $\ff_1 =
\head{\ff_1\ff_2}$, hence $\f(\ff_1) = \head{\f(\ff_1\ff_2)}$ by hypothesis. This means that the normal form
of~$\f(\ff_1\ff_2)$ is $(\f(\ff_1), \gg)$ for some~$\gg$ satisfying $\f(\ff_1\ff_2) = \f(\ff_1)\gg$. Now
$\f(\ff_2)$ is such a morphism~$\gg$, and, by~$\LGGi$, it is the only one. So the normal form
of~$\f(\ff_1\ff_2)$ is $(\f(\ff_1), \f(\ff_2))$, and $\CCC$ is regular.
\end{proof}

\begin{prop}
\label{P:CriterionTer}
Assume that $\CCC$ is a left-Garside category~$\CCC$. Then two sufficient conditions for~$\CCC$ to be
regular are

$(i)$ The functor~$\f$ preserves left-coprimeness of simple morphisms: for $\ff, \gg$ simple with
$\source{\ff} = \source{\gg}$,
\begin{equation}
\label{E:PhiCoprime}
\gcd(\ff, \gg) = 1
\mbox{\quad implies \quad}
\gcd(\f(\ff), \f(\gg)) = 1.
\end{equation}

$(ii)$ The functor~$\f$ preserves the gcd operation on simple morphisms:  for $\ff,
\gg$ simple with $\source{\ff} = \source{\gg}$,
\begin{equation}
\gcd(\f(\ff), \f(\gg)) = \f(\gcd(\ff, \gg)),
\end{equation}
and, moreover, $\f(\ff)$ is nontrivial whenever $\ff$ is nontrivial.
\end{prop}

\begin{proof}
Assume~$(i)$. Let $(\ff, \gg)$ be normal.
By Lemma~\ref{L:Coprime}, we have $\gcd(\ff^*, \gg) =\nobreak 1$. By~\eqref{E:PhiCoprime}, we
deduce $\gcd(\f(\ff^*), \f(\gg)) = 1$. By Lemma~\ref{L:Dual}, this equality is also
$\gcd(\f(\ff)^*, \f(\gg)) = 1$, which, by Lemma~\ref{L:Coprime} again, means that
$(\f(\ff), \f(\gg))$ is normal. Hence $\CCC$ is regular.

On the other hand, it is clear that $(ii)$ implies~$(i)$.
\end{proof}

\section{Self-distributivity}
\label{S:LD}

We quit general left-Garside categories, and turn to the description of one particular
example, namely a certain category (two categories actually) associated with the left self-distributive law. The
latter is the algebraic law
\begin{equation}
\label{E:LD}
\tag{\hbox{LD}}
\xx(\yy\zz) = (\xx\yy)(\xx\zz)
\end{equation}
extensively investigated in~\cite{Dgd}.

We first review some basic results about this law and the associated free LD-systems, \ie, the binary systems
that obey the LD-law. The key notion is the notion of an LD-expansion, with two
derived categories~$\CLDm$ and~$\CLD$ that will be our main subject of investigation from now on.

\subsection{Free LD-systems}
\label{S:Free}

For each algebraic law (or family of algebraic laws), there exist universal objects in the
category of structures that satisfy this law, namely the free systems. Such structures can
be uniformly described as the quotient of some absolutely free structures under a
convenient congruence.

\begin{defi}
We let~$\TT_\nn$ be the set of all bracketed expressions involving variables~$\xx_1, ..., \xx_\nn$, \ie,
the closure of~$\{\xx_1, ..., \xx_\nn\}$ under $\tt_1 \op \tt_2 =
(\tt_1)(\tt_2)$. We use~$\TT$ for the union of all sets~$\TT_\nn$. Elements of~$\TT$ are called
\emph{terms}.
\end{defi}

Typical terms are $\xx_1$, $\xx_2
\op \xx_1$, $\xx_3 \op (\xx_3 \op \xx_1)$, etc. It is convenient to think of terms as
rooted binary trees with leaves indexed by the variables: the trees associated with the
previous terms are 
\vrule width0pt height7mm depth6mm
\begin{picture}(4,1)(-1,-1.5)
\put(0,0){$\scriptscriptstyle\bullet$}
\put(-1,-3){$\xx_1$}
\end{picture},
\begin{picture}(7,1)(-1,0)
\put(0,0){\includegraphics{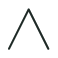}}
\put(-1,-3){$\xx_2$}
\put(3,-3){$\xx_1$}
\end{picture}, and
\begin{picture}(10,1)(-2,2)
\put(0,0){\includegraphics{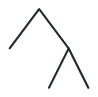}}
\put(-1,1){$\xx_3$}
\put(3,-3){$\xx_3$}
\put(7,-3){$\xx_1$}
\end{picture}
, respectively. The system $(\TT_\nn, \op)$ is the absolutely free
system (or algebra) generated by~$\xx_1, ..., \xx_\nn$, and every binary system
generated by $\nn$~elements is a quotient of this system. So is in particular the free LD-system of
rank~$\nn$.

\begin{defi}
We denote by~$\eLD$ the least congruence (\ie, equivalence relation compatible with
the product) on~$(\TT_\nn, \op)$ that contains all pairs of the form
$$(\tt_1 \op (\tt_2 \op \tt_3), (\tt_1 \op \tt_2) \op (\tt_1 \op \tt_3)).$$
Two terms~$\tt, \tt'$ satisfying $\tt \eLD \tt'$ are called \emph{LD-equivalent}.
\end{defi}

The following result is then standard.

\begin{prop}
\label{P:Free}
For each~$\nn \le \infty$, the binary system $(\TT_\nn\quot\eLD, \op)$ is a free
LD-system based on $\{\xx_1, ..., \xx_\nn\}$.
\end{prop}

\subsection{LD-expansions}

The relation~$\eLD$ is a complicated object, about which many questions remain open.
In order to investigate it, it proved useful to introduce the subrelation of~$\eLD$ that
corresponds to applying the LD-law in the expanding direction only.

\begin{defi}
Let $\tt, \tt'$ be terms. We say that $\tt'$ is an \emph{atomic LD-expansion} of~$\tt$,
denoted $\tt \exp^1 \tt'$, if $\tt'$ is obtained from~$\tt$ by replacing some subterm of
the form $\tt_1 \op (\tt_2 \op \tt_3)$ with the corresponding term $(\tt_1 \op \tt_2) \op
(\tt_1 \op \tt_3)$. We say that $\tt'$ is an \emph{LD-expansion} of~$\tt$, denoted $\tt
\exp \tt'$, if there exists a finite sequence of terms $\tt_0, ..., \tt_\pp$ satisfying $\tt_0 =
\tt$, $\tt_\pp = \tt'$, and $\tt_{\ii-1} \exp^1 \tt_\ii$ for $1 \le  \ii \le \pp$.
\end{defi}

By definition, being an LD-expansion implies being LD-equivalent, but the converse is not
true. For instance, the term $(\xx \op \xx) \op (\xx \op \xx)$ is an (atomic) LD-expansion
of $\xx \op (\xx \op \xx)$, but the latter is not an LD-expansion of the former. However,
it should be clear that $\eLD$ is generated by~$\exp$, so that
two terms~$\tt, \tt'$ are LD-equivalent if and only if there exists a finite zigzag $\tt_0,
\tt_1, ..., \tt_{2\pp}$ satisfying $\tt_0 = \tt$, $\tt_{2\pp} = \tt'$, and $\tt_{\ii-1} \exp
\tt_\ii \antiexp \tt_{\ii+1}$ for each odd~$\ii$.

The first nontrivial result about LD-equivalence is that the previous zigzags may always be
assumed to have length two. 

\begin{prop} \cite{Dep}
\label{P:Confluence}
Two terms are LD-equivalent if and only if they admit a common LD-expansion.
\end{prop}

This result is similar to the property that, if a monoid~$\MM$ satisfies Ore's condi\-tions---as the braid
monoid~$\BP\nn$ does for instance---then every element in the universal group of~$\MM$ can be expressed as
a fraction of the form~$\aa\bb\inv$ with $\aa, \bb$ in~$\MM$. Proposition~\ref{P:Confluence} plays a
fundamental role in the sequel, and we need to recall some elements of its proof.

\begin{defi}\cite{Dep}
First, a binary operation~$\OP$ on terms is recursively  defined by
\begin{equation}
\tt \OP \xx_\ii = \tt \op \xx_\ii, 
\quad
\tt \OP (\tt_1 \op \tt_2) = (\tt \OP \tt_1) \op (\tt \OP \tt_2).
\end{equation}
Next, for each term~$\tt$, the term~$\f(\tt)$
 is recursively defined by\footnote{In~\cite{Dep} and~\cite{Dgd}, 
$\partial$ is used instead of~$\phi$, an inappropriate notation in the current context.}
\begin{equation}
\f(\xx_\ii) = \xx_\ii
\quad
\f(\tt_1 \op \tt_2) = \f(\tt_1) \OP \f(\tt_2).
\end{equation}
\end{defi}

The idea is that $\tt \OP \tt'$ is obtained by distributing~$\tt$  everywhere in~$\tt'$
once. Then $\f(\tt)$ is the image of~$\tt$ when $\op$ is replaced with~$\OP$ everywhere
in the unique expression of~$\tt$ in terms of variables. Examples are given in
Figure~\ref{F:Phi}. A straightforward induction shows that $\tt \OP \tt'$ is always an
LD-expansion of~$\tt \op \tt'$ and, therefore, that $\f(\tt)$ is an LD-expansion of~$\tt$.

\begin{figure}[htb]
\begin{picture}(75,16)(0,2)
\put(0,0){\includegraphics[scale = 1]{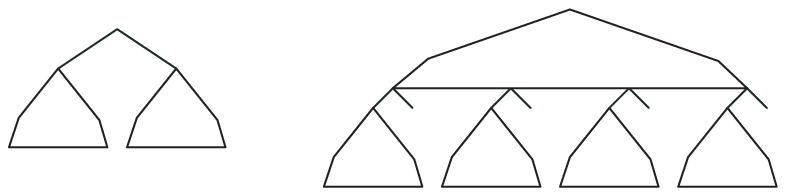}}
\put(-5,10){$\f\Bigg($}
\put(24,10){$\Bigg)\ =$}
\put(4.5,7){$\tt_1$}
\put(16.5,7){$\tt_2$}
\put(34.5,2.5){$\scriptstyle\f(\tt_1)$}
\put(46.5,2.5){$\scriptstyle\f(\tt_1)$}
\put(58.5,2.5){$\scriptstyle\f(\tt_1)$}
\put(70.5,2.5){$\scriptstyle\f(\tt_1)$}
\put(53.5,13.5){$\f(\tt_2)$}
\end{picture}
\caption{\sf\smaller The fundamental LD-expansion~$\f(\tt)$ of a term~$\tt$, recursive definition:  $\f(\tt_1
\op \tt_2)$ is obtained by distributing~$\f(\tt_1)$ everywhere in~$\f(\tt_2)$.}
\label{F:Phi}
\end{figure}

\begin{figure}[htb]
\begin{picture}(52,12)(0,0)
\put(0,0.5){\includegraphics[scale = 1]{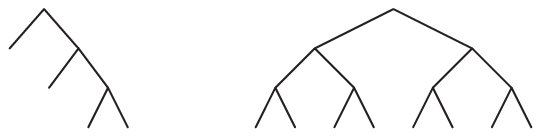}}
\put(-7,4){$\f\Bigg($}
\put(15,4){$\Bigg)\ =$}
\put(-1,6){$\xx_1$}
\put(3,2){$\xx_2$}
\put(7,-2){$\xx_3$}
\put(11,-2){$\xx_4$}
\put(24,-2){$\xx_1$}
\put(28,-2){$\xx_2$}
\put(32,-2){$\xx_1$}
\put(36,-2){$\xx_3$}
\put(40,-2){$\xx_1$}
\put(44,-2){$\xx_2$}
\put(48,-2){$\xx_1$}
\put(52,-2){$\xx_4$}
\end{picture}
\caption{\sf\smaller The fundamental LD-expansion~$\f(\tt)$ of a term~$\tt$: an example; 
$\tt = \xx_1 \op (\xx_2 \op (\xx_3 \op \xx_4))$ implies
$\f(\tt) = \xx_1 \OP (\xx_2  \OP (\xx_3  \OP \xx_4))$.}
\label{F:PhiBis}
\end{figure}

The main step for establishing Proposition~\ref{P:Confluence} consists in proving that
$\f(\tt)$ plays with respect to atomic LD-expansions a role similar to Garside's
fundamental braid~$\D_\nn$ with respect to Artin's generators~$\sig\ii$---which makes it natural to
call~$\f(\tt)$ the \emph{fundamental LD-expansion} of~$\tt$.

\begin{lemm}
\label{L:Confluence} \cite{Dep} \cite[Lemmas V.3.11 and V.3.12]{Dgd} 
$(i)$ The term~$\f(\tt)$ is an LD-expansion of each atomic LD-expansion \nobreak of~$\tt$.

$(ii)$ If $\tt'$ is an LD-expansion of~$\tt$, then $\f(\tt')$ is an LD-expansion of~$\f(\tt)$.
\end{lemm}

\begin{proof}[Sketch of proof]
One uses induction on the size of the involved terms. Once
Lemma~\ref{L:Confluence} is established, an easy induction on~$\dd$ shows that, if 
there exists a length~$\dd$ sequence of atomic LD-expansions connecting~$\tt$ to~$\tt'$,
then $\f^\dd(\tt)$ is an LD-expansion of~$\tt'$. Then a
final induction on the length of a zigzag connecting~$\tt$ to~$\tt'$ shows that, if $\tt$
and $\tt'$ are LD-equivalent, then
$\f^\dd(\tt)$ is an LD-expansion of~$\tt'$ for sufficiently large~$\dd$ (namely for
$\dd$ at least the number of ``zag''s in the zigzag).
\end{proof}

\subsection{The category $\CLDm$}
\label{S:LCDm}

A category (and a quiver) is naturally associated with every graph, and the previous
results invite to introduce the category associated with the LD-expansion relation~$\exp$.

\begin{defi}
We denote by~$\CLDm$ the category whose objects are terms, and whose morphisms are
pairs of terms $(\tt, \tt')$ satisfying $\tt \exp \tt'$.
\end{defi}

By construction, the category~$\CLDm$ is left- and right-cancellative, and
Proposiion~\ref{P:Confluence} means that any two morphisms of~$\CLDm$ with the same
source admit a common right-multiple.  Moreover, a natural candidate for being a
left-Garside map is obtained by defining $\D(\tt) = (\tt, \f(\tt))$ for each term~$\tt$.

\begin{ques}
\label{Q:EmbConj1}
Is $\CLDm$ a left-Garside category?
\end{ques}

Question~\ref{Q:EmbConj1} is currently open. We shall see in Section~\ref{S:EmbConj}
that it is one of the many forms of the so-called Embedding Conjecture. The
missing part is that we do not know that \emph{least} common multiples exist
in~$\CLDm$, the problem being that we have no method for proving that a common
LD-expansion of two terms is possibly a least common LD-expansion.

\section{The monoid~$\MLD$ and the category~$\CLD$}
\label{S:CLD}

The solution for overcoming the above difficulty consists in developing a more precise
study of LD-expansions that takes into account the position where the LD-law is applied.
This leads to introducing a certain monoid~$\MLD$ whose elements provide natural
labels for LD-expansions, and, from there, a new category~$\CLD$, of which $\CLDm$ is a
projection. This category~$\CLD$ is the one on which a left-Garside structure will be
proved to exist.

\subsection{Labelling LD-expansions}
\label{S:Labelling}

By definition, applying the LD-law to a term~$\tt$ means selecting some subterm
of~$\tt$ and replacing it with a new, LD-equivalent term. When terms are viewed as
binary rooted trees, the position of a subterm can be specified by describing the
path that connects the root of the tree to the root of the considered subtree, hence
typically by a binary address, \ie, a finite sequence of~$0$'s and~$1$'s, according to the
convention that $0$ means ``forking to the left'' and $1$ means ``forking to the right''.
Hereafter, we use $\Add$ for the set of all such addresses, and $\ea$ for the empty
address, which corresponds to the position of the root in a tree.

\begin{nota}
For $\tt$ a term and $\a$ an address, we denote by~$\sub\tt\a$ the subtree of~$\tt$
whose root has address~$\a$, if it exists, \ie, if $\a$ is short enough.
\end{nota}

So, for instance, if $\tt$ is the tree $\xx_1 \op (\xx_2 \op \xx_3)$, we have $\sub\tt0 =
\xx_1$, $\sub\tt{10} = \xx_2$, whereas $\sub\tt{00}$ is not defined, and $\sub\tt\ea =
\tt$ holds, as it holds for every term.

\begin{defi}
\label{D:LDAction}
(See Figure~\ref{F:Apply}.)
We say that $\tt'$ is a \emph{$\DD\a$-expansion} of~$\tt$, denoted $\tt' = \tt \act
\DD\a$, if
$\tt'$ is the atomic LD-expansion of~$\tt$ obtained by applying LD at the position~$\a$,
\ie, replacing the subterm~$\sub\tt\a$, which is $\sub\tt0 \op (\sub\tt{10} \op
\sub\tt{11})$, with the term $(\sub\tt0 \op \sub\tt{10}) \op (\sub\tt0 \op \sub\tt{11})$.
\end{defi}

\begin{figure}[tb]
\begin{picture}(85,32)(0,-4)
\put(0,0){\includegraphics{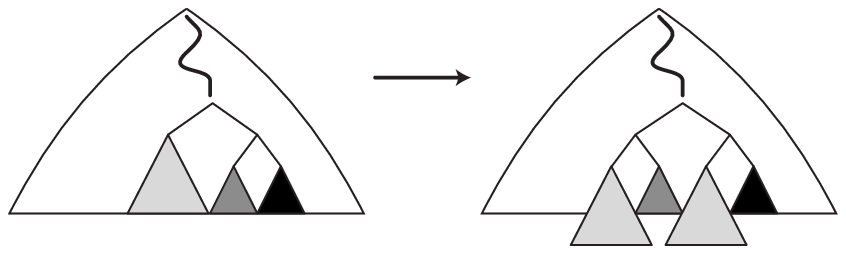}}
\put(18,25){$\tt$}
\put(39,19){$\DD\a$}
\put(65,25){$\tt\act\DD\a$}
\put(22,15){$\a$}
\put(70,15){$\a$}
\put(14,0){$\sub\tt{\a\!0}$}
\put(20,0){$\sub\tt{\a\!1\!0}$}
\put(25.5,0){$\sub\tt{\a\!1\!1}$}
\put(58,-3){$\sub\tt{\a\!0}$}
\put(63,-5){$\sub\tt{\a\!1\!0}$}
\put(68.5,-3){$\sub\tt{\a\!0}$}
\put(74,-5){$\sub\tt{\a\!1\!1}$}
\put(75,-2.5){$\uparrow$}
\put(65,-2.5){$\uparrow$}
\end{picture}
\caption{\sf\smaller  Action of~$\DD\a$ to a term~$\tt$: the LD-law is applies to
expand~$\tt$ at position~$\a$, \ie, to replace the subterm~$\sub\tt\a$, which is
$\sub\tt{\a\0} \op (\sub\tt{\a\1\0}\op\sub\tt{\a\1\1})$, with
$(\sub\tt{\a\0} \op \sub\tt{\a\1\0})\op (\sub\tt{\a\0} \op \sub\tt{\a\1\1})$; in other
words, the light grey subtree is duplicated and distributed to the left of the dark grey
and black subtrees.}
\label{F:Apply}
\end{figure}

By construction, every atomic LD-expansion is a $\DD\a$-expansion for a unique~$\a$.
The idea is to use the letters~$\DD\a$ as labels for LD-expansions. As arbitrary
LD-expansions are compositions of finitely many atomic LD-expansions, hence of
$\DD\a$-expansions, it is natural to use finite sequences of~$\DD\a$ to
label LD-expansions. In other words, we extend the (partial) action of~$\DD\a$
on terms into a (partial) action of finite sequences of~$\DD\a$'s. Thus, for instance, we
write
$$\tt' = \tt \act \DD\a \DD\b \DD\g$$
to indicate that $\tt'$ is the LD-expansion of~$\tt$ obtained by successively applying the
LD-law (in the expanding direction) at the positions~$\a$, then~$\b$, then~$\g$.

\begin{lemm}
Definition~\ref{D:LDAction} gives a partial action of the free monoid $\{\DD\a\mid\a\in\Add\}^*$ on~$\TT$
(the set of terms), in the sense of Definition~\ref{D:Action}.
\end{lemm}

\begin{proof}
Conditions~$(i)$ and~$(ii)$ of Definition~\ref{D:Action} follow from the construction. The point is to
prove~$(iii)$, \ie, to prove that, if $\ww_1, ..., \ww_\nn$ are arbitrary finite sequences of letters~$\DD\a$,
then there exists at least one term~$\tt$ such that $\tt \act\ww_\ii$ is defined for each~$\ii$. This is what
\cite[Proposition~VII.1.21]{Dgd} states.
\end{proof}

\subsection{The monoid~$\MLD$}
\label{S:MLD}

There exist clear connections between the action of various~$\DD\a$'s: different sequences may
lead to the same transformations of trees. Our approach will consist in identifying
a natural family of such relations and introducing the monoid presented by these
relations. 

\begin{lemm}
\label{L:RLD}
For all $\a, \b, \g$, the following pairs have the same action on trees:

$(i)$ $\DD{\a0\b} \DD{\a1\g}$ and $\DD{\a1\g} \DD{\a0\b}$;
\hfill(``parallel case'')\hspace{1cm}

$(ii)$ $\DD{\a0\b} \DD{\a}$ and $\DD\a \DD{\a00\b} \DD{\a10\b}$;
\hfill(``nested case 1'')\hspace{1cm}

$(iii)$ $\DD{\a10\b} \DD{\a}$ and $\DD\a \DD{\a01\b}$;
\hfill(``nested case 2'')\hspace{1cm}

$(iv)$ $\DD{\a11\b} \DD{\a}$ and $\DD\a \DD{\a11\b}$;
\hfill(``nested case 3'')\hspace{1cm}

$(v)$ $\DD{\a1} \DD\a$ and $\DD{\a\1} \DD{\a} \DD{\a1} \DD{\a\0}$.
\hfill(``critical case'')\hspace{1cm}\null
\end{lemm}

\begin{proof}[Sketch of proof]
The commutation relation of the parallel case is clear, as the transformations involve disjoint
subterms. The nested cases are commutation relations as well, but, because one of the involved subtree is
nested in the other, it may be moved, and even possibly duplicated when the main expansion is performed, so
that the nested expansion(s) correspond to different names before and after the main expansion. Finally, the
critical case is specific to the LD-law, and there is no way to predict it except the verification, see
Figure~\ref{F:Critical}.
\end{proof}

\begin{figure}[tb]
\begin{picture}(101,37)(0,1)
\put(0,0){\includegraphics{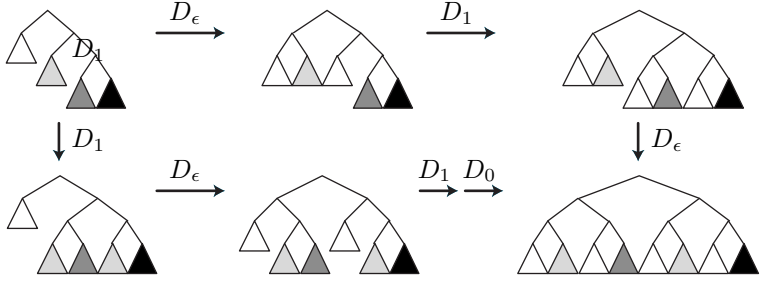}}
\put(10,18){$\DD\1$}
\put(23,35){$\DD\ea$}
\put(10,30){$\DD\1$}
\put(59,35){$\DD\1$}
\put(87,18){$\DD\ea$}
\put(23,14){$\DD\ea$}
\put(56,14){$\DD\1$}
\put(62,14){$\DD\0$}
\end{picture}
\caption{\sf\smaller  Relations between $\DD\a$-expansions: the critical case. We read
that the action of $\DD\ea \DD\1 \DD\ea$ and $\DD\1 \DD\ea \DD\1 \DD\0$
coincide.}
\label{F:Critical}
\end{figure}

\begin{defi}
\label{D:RLD}
Let $\RLD$ be the family of all relations of Lemma~\ref{L:RLD}. We define $\MLD$ to be
the monoid $\Mon{\{\DD\a \mid \a \in \Add\} \mid \RLD}$.
\end{defi}

Lemma~\ref{L:RLD} immediately implies

\begin{prop}
The partial action of the free monoid $\{\DD\a \mid \a\in\Add\}^*$ on terms induces a well defined partial
action of the monoid~$\MLD$.
\end{prop}

For $\tt$ a term and $\aa$ in~$\MLD$, we shall naturally
denote by $\tt \act \aa$ the common value of~$\tt \act \ww$ for all sequences~$\ww$
of~$\DD\a$ that represent~$\aa$. 

\begin{rema}
In this way, each LD-expansion receives a label that is an element of~$\MLD$, thus
becoming a labelled LD-expansion. However, we do not claim that a labelled LD-expansion
are the same as an LD-expansion. Indeed, we do not claim that the relations of
Lemma~\ref{L:RLD} exhaust all possible relations between the action of the~$\DD\a$'s
on terms. A priori, it might be that different elements of~$\MLD$ induce the same action
on terms, so that one pair $(\tt, \tt')$ might correspond to several labelled
expansions with different labels. As we shall see below, the uniqueness of the labelling is another form of the
above mentioned Embedding Conjecture.
\end{rema}

\subsection{The category $\CLD$}
\label{SS:CLD}

We are now ready to introduce our main subject of interest, namely the category~$\CLD$
of labelled LD-expansions. The starting point is the same as for~$\CLDm$, but the
difference is that, now, we explicitly take into account the way of expanding the source
is expanded into the target.

\begin{defi}
We denote by~$\CLD$ the category whose objects are terms, and whose morphisms are
triples $(\tt, \aa, \tt')$ with $\aa$ in~$\MLD$ and $\tt \act \aa = \tt'$.
\end{defi}

In other words, $\CLD$ is the category associated with the partial action of~$\MLD$ on terms, in the sense of
Section~\ref{D:ActionCat}. We recall our convention that, when the morphisms of a category are triples, the
source is the first entry, and the target is the the last entry. So, for instance, a typical morphism
in~$\CLD$ is the triple 
$$
\vrule width0pt height7mm depth4mm
\begin{picture}(50,5)(0,0)
\put(0,0){\Bigg(}
\put(3,-4){\includegraphics{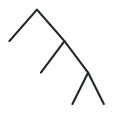}}
\put(15,0){, $\DD\ea \DD\1$ , }
\put(30,-4){\includegraphics{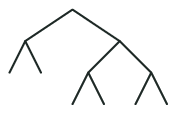}}
\put(47,0){\Bigg),}
\end{picture},
$$
whose source is the term $\xx\op(\xx\op(\xx\op\xx))$ (we the default convention that unspecified variables
means some fixed variable~$\xx$), and whose target is the term $(\xx\op\xx)\op
((\xx\op\xx) \op(\xx\op\xx))$.

\subsection{The element~$\D_\tt$}
\label{S:Delta}

We aim at proving that the category~$\CLD$ is a left-Garside
category. To this end, we need to define the $\D$-morphisms. As planned in
Section~\ref{S:LCDm}, the latter will be constructed using the LD-expansions~$(\tt,
\f(\tt))$. Defining a labelled version of this expansion means fixing some
canonical way of expanding a term~$\tt$ into the corresponding term~$\f(\tt)$. A
natural solution then exists, namely following the recursive definition of the
operations~$\OP$ and~$\f$.

For~$\ww$ a word in the letters~$\DD\a$, we denote by $\sh0(\ww)$ the word obtained
by replacing each letter~$\DD\a$ of~$\ww$ with the corresponding letter~$\DD{0\a}$,
\ie, by \emph{shifting} all indices by~$0$. Similarly, we denote by~$\sh\g(\ww)$
the word obtained by appending~$\g$ on the left of each address in~$\ww$. The
LD-relations of Lemma~\ref{L:RLD} are invariant under shifting: if $\ww$ and
$\ww'$ represent the same element~$\aa$ of~$\MLD$, then, for each~$\g$, the words
$\sh\g(\ww)$ and $\sh\g(\ww')$ represent the same element, naturally
denoted~$\sh\g(\aa)$, of~$\MLD$. By construction, $\sh\g$ is an endomorphism of the
monoid~$\MLD$. For each~$\aa$ in~$\MLD$, the action of~$\sh\g(\aa)$ on a
term~$\tt$ corresponds to the action of~$\aa$ to the $\g$-subterm of~$\tt$: so, for
instance, if $\tt' = \tt \act \aa$ holds, then $\tt' \op \tt_1 = (\tt \op \tt_1) \act
\sh0(\aa)$ holds as well, since the $0$-subterm of~$\tt \op \tt_1$ is~$\tt$, whereas that
of~$\tt' \op \tt_1$ is~$\tt'$.

\begin{defi}
\label{D:Delta}
For each term~$\tt$, the elements~$\d_\tt$ and~$\D_\tt$ of~$\MLD$ are
defined by the recursive rules
\begin{gather}
\label{ED:delta}
\d_\tt = 
\begin{cases}
1
& \mbox{for $\tt$ of size~$1$, \ie, when $\tt$ is a variable~$\xx_\ii$,}\\
\DD\ea \cdot \sh0(\d_{\tt_0}) \cdot \sh1(\d_{\tt_1})
&\mbox{for $\tt = \tt_0 \op \tt_1$.}
\end{cases}\\
\label{ED:Delta}
\D_\tt = 
\begin{cases}
1
& \mbox{for $\tt$ of size~$1$,}\\
\sh\0(\D_{\tt_\0}) \cdot \sh\1(\D_{\tt_\1}) \cdot \d_{\f(\tt_\1)}
&\mbox{for $\tt = \tt_\0 \op \tt_\1$.}
\end{cases}
\end{gather}
\end{defi}

\begin{exam}
Let $\tt$ be $\xx \op (\xx \op (\xx \op \xx))$. Then $\sub\tt\0$ is $\xx$, and, therefore,
$\D_{\sub\tt\0}$ is~$1$. Next, $\sub\tt\1$ is $\xx \op (\xx \op \xx)$, so
\eqref{ED:Delta} reads $\D_\tt = \sh\1(\D_{\sub\tt\1}) \cdot \d_{\f(\sub\tt\1)}$.
Then $\f(\sub\tt\1)$ is $(\xx \op \xx) \op (\xx \op \xx)$. Applying \eqref{ED:delta},
we obtain 
$$\d_{\f(\sub\tt\1)} = \DD\ea \cdot \sh\0(\d_{\xx \op \xx}) \cdot
\sh\1(\d_{\xx \op \xx}) = \DD\ea \DD\0 \DD\1.$$
On the other hand, using \eqref{ED:Delta} again, we fing 
$$\D_{\sub\tt\1} = \sh\0(\D_\xx) \cdot \sh\1(\D_{\xx \op \xx}) \cdot \d_{\xx \op \xx}
= 1 \cdot 1 \cdot \DD\ea = \DD\ea,$$
and, finally, we obtain $\D_\tt = \DD\1 \DD\ea \DD\0 \DD\1$. According to the
defining relations of the monoid~$\MLD$, this element is also $\DD\ea \DD\1 \DD\ea$. Note the
compatibility of the result with the examples of Figures~\ref{F:PhiBis} and~\ref{F:Critical}.
\end{exam}

\begin{lemm}
\label{L:Delta}
For all terms~$\tt_0, \tt$, we have
\begin{gather}
\label{E:delta}
(\tt_0 \op \tt) \act \d_\tt = \tt_0 \OP \tt,\\
\label{E:Delta}
\tt \act \D_\tt = \f(\tt).
\end{gather}
\end{lemm}

The proof is an easy inductive verification.

\subsection {Connection with braids}
\label{S:Connection}

Before investigating the category~$\CLD$ more precisely, we describe the simple
connection existing between the category~$\CLD$ and the positive braid category~$\BBB$
of Example~\ref{X:Braid}.

\begin{lemm}
\label{L:Proj}
Define $\pi : \{\DD\a \mid \a \in \Add\} \to \{\sig\ii \mid \ii\ge 1\} \cup\{1\}$ by
\begin{equation}
\label{E:Proj}
\pi(\DD\a) = 
\begin{cases}
\sig{\ii+1}
&\mbox{ if $\a$ is the address $\1^\ii$, \ie, $\1\1... \1$, $\ii$~times~$\1$},\\
1
&\mbox{ otherwise.}
\end{cases}
\end{equation}
Then $\pi$ induces a surjective monoid homomorphism of~$\MLD$ onto~$\BP\infty$.
\end{lemm}

\begin{proof}
The point is that each LD-relation of Lemma~\ref{L:RLD} projects
under~$\pi$ onto a braid equivalence. All relations involving addresses that contain at
least one~$\0$ collapse to mere equalities. The remaining relations are 
$$\DD{\1^{\ii}} \DD{\1^{\jj}} = \DD{\1^{\jj}} \DD{\1^{\ii}}
\mbox{\quad with $\jj \ge \ii + 2$},$$
which projects to the valid braid relation
$\sig{\ii -1} \sig{\jj - 1} = \sig{\jj - 1} \sig{\ii -1}$, 
and 
$$\DD{\1^{\ii}} \DD{\1^{\jj}} \DD{\1^{\ii}} = \DD{\1^{\jj}} \DD{\1^{\ii}}
\DD{\1^{\jj}} \DD{\1^{\ii}\0},
\mbox{\quad with $\jj = \ii + 1$},$$
which projects to the not less valid braid relation
$\sig{\ii -1} \sig{\jj - 1} \sig{\ii -1} = \sig{\jj - 1} \sig{\ii -1} \sig{\jj - 1}$.
\end{proof}

We introduced a category~$\CMX$ for each monoid~$\MM$ partially acting on~$\XX$ in
Definition~\ref{D:ActionCat}. The braid category~$\BBB$ and our current category~$\CLD$
are of these type. For such categories, natural functors arise from morphisms between the involved monoids,
and we fix the following notation.

\begin{defi}
Assume that $\MM, \MMM$ are monoids acting on sets~$\XX$ and~$\XXX$, respectively. A
morphism~$\varphi{:}\,\MM {\to} \MMM$ and a map~$\psi{:}\,\XX {\to} \XXX$ are called
\emph{compatible} if 
\begin{equation}
\label{E:Compat}
\psi(\xx\act\aa) = \psi(\xx) \act \varphi(\aa)
\end{equation}
holds whenever $\xx\act\aa$ is defined. Then, we denote by~$\MAct\varphi\psi$ the functor
of~$\CAct\MM\XX$ to~$\CAct\MMM\XXX$ that coincides with~$\psi$ on objects and maps $(\xx, \aa,
\yy)$ to~$(\psi(\xx), \varphi(\aa), \psi(\yy))$.
\end{defi}

\begin{prop}
\label{P:Proj}
Define the \emph{right-height}~$\RH(\tt)$ of a term~$\tt$ by $\RH(\xx_\ii) =\nobreak 0$ and
$\RH(\tt_\0 \op \tt_\1) = \RH(\tt_\1) + 1$. Then the morphism~$\pi$ of~\eqref{E:Proj} is compatible
with~$\RH$, and $\MAct\pi\RH$ is a surjective functor of~$\CLD$ onto~$\BBB$. 
\end{prop}

The parameter~$\RH(\tt)$ is  the length of the rightmost branch in~$\tt$ viewed as a tree
or, equivalently, the number of final~$)$'s in~$\tt$ viewed as a bracketed expression. 

\begin{proof}
Assume that $(\tt, \aa, \tt')$ belongs to~$\Hom(\CLD)$. Put $\nn = \RH(\tt)$. The LD-law preserves the
right-height of terms, so we have $\RH(\tt') = \nn$ as well. The hypothesis that $\tt \act \aa$ exists implies
that the factors~$\DD{\1^\ii}$ that occur in some (hence in every) expression of~$\aa$ satisfy $\ii < \nn-1$.
Hence $\pi(\aa)$ is a braid of~$\BP\nn$, and $\nn \act \pi(\aa)$ is defined. Then the compatibility
condition~\eqref{E:Compat} is clear, $\MAct\pi\RH$ is a functor of~$\CLD$ to~$\BBB$. 

Surjectivity is clear, as each braid~$\sig\ii$ belongs to the image of~$\pi$.
\end{proof}

Moreover, a simple relation connects the elements~$\D_\tt$ of~$\MLD$ and the
braids~$\D_\nn$.

\begin{prop}
\label{P:ProjDelta}
We have $\pi(\D_\tt) = \D_\nn$ whenever~$\tt$ has right-height~$\nn \ge 1$.
\end{prop}

\begin{proof}
We first prove that $\RH(\tt) = \nn$ implies
\begin{equation}
\label{E:delta}
\pi(\d_\tt) = \sig1\sig2 ... \sig\nn
\end{equation}
using induction on the size of~$\tt$. If~$\tt$ is a variable, we have
$\RH(\tt) = 0$ and $\d_\tt = 1$, so the equality is clear. Otherwise, write $\tt = \tt_\0 \op
\tt_\1$. By definition, we have
$$\d_\tt = \DD\ea \cdot \sh\0(\d_{\tt_\0}) \cdot \sh\1(\d_{\tt_\1}).$$
Let $\sh{}$ denote the endomorphism of~$\BP\infty$ that maps~$\sig\ii$ to~$\sig{\ii+1}$ for
each~$\ii$. Then $\pi$ collapses every term in the image of~$\sh\0$, and 
$\pi(\sh\1(\aa)) = \sh{}(\pi(\aa))$ holds for each~$\aa$ in~$\MLD$. Hence, using the induction hypothesis
$\pi(\d_{\tt_\1}) = \sig1...\sig{\nn-1}$, we deduce
$$\pi(\d_\tt) = \sig1 \cdot 1 \cdot \sh{}(\sig1...\sig{\nn-1})
= \sig1...\sig{\nn},$$
which is~\eqref{E:delta}. 
Put~$\D_0 = 1$ (= $\D_1$). We prove that $\RH(\tt) = \nn$ implies
$\pi(\D_\tt) = \D_\nn$ for~$\nn \ge 0$, using induction on the size of~$\tt$ again. If~$\tt$ is a variable, we
have $\nn = 0$ and $\D_\tt = 1$, as expected. Otherwise, write
$\tt = \tt_\0 \op \tt_\1$. The definition gives
$$\D_\tt = \sh\0(\D_{\tt_\0}) \cdot \sh\1(\D_{\tt_\1}) \cdot \d_{\f(\tt_\1)}.$$
As above, $\pi$ collapses the term in the image of~$\sh\0$, and it transforms~$\sh\1$
into~$\sh{}$. Hence, using the induction hypothesis $\pi(\D_{\tt_\1}) = \D_\nno$ and
\eqref{E:delta} for~$\f(\tt_1)$, whose right-height is that of~$\tt_\1$, we obtain
$$\pi(\D_\tt) = 1 \cdot \sh{}(\D_\nno) \cdot \sig1 \sig2 \,...\, \sig\nno = \D_\nn.
\eqno{\square}$$
\def\qed{\relax}
\end{proof}

\section{The main results}
\label{S:Main}

We can now state the two main results of this paper.

\begin{thrm}
\label{T:Main}
For each term~$\tt$, put $\D(\tt) = (\tt, \D_\tt, \f(\tt))$. Then $\CLD$ is a left-Garside category with 
left-Garside map~$\D$, and $\MAct\pi\RH$ is a surjective right-lcm
preserving functor of~$\CLD$ onto the positive braid category~$\BBB$.
\end{thrm}

\begin{thrm}
\label{T:Embedding}
Unless the category~$\CLD$ is not regular, the Embedding Conjecture of~\cite[Chapter~IX]{Dgd} is true.
\end{thrm}

\subsection{Recognizing left-preGarside monoids}
\label{S:Recognizing}

Owing to Proposition~\ref{P:Criterion} and to the construction of~$\CLD$ from the partial action
of the monoid~$\MLD$ on terms, the first part of Theorem~\ref{T:Main} is a direct
consequence of

\begin{prop}
\label{P:Main}
The monoid~$\MLD$ equipped with its partial action on terms via self-distributivity is a locally left-Garside 
monoid with associated left-Garside sequence~$(\D_\tt)_{\tt \in \TT}$.
\end{prop}

This is the result we shall prove now. The first step is to prove that $\MLD$ is left-preGarside. To do it, we
appeal to general tools that we now describe. As for~$\LG$, we have an easy sufficient condition when the
action turns out to be monotonous in the following sense.

\goodbreak

\begin{prop}
\label{P:Mass}
Assume that $\MM$ has a partial action on~$\XX$ and there exists a map
$\mu:
\XX
\to \Nat$ such that $\aa \not= 1$ implies
$\mu(\xx\act\aa) > \mu(\xx)$. Then $\MM$ satisfies~$\LG$.
\end{prop}

\begin{proof}
Assume that $(\aa_1, ..., \aa_\ell)$ is a $\div$-increasing sequence
in~$\Div(\aa)$. By definition of a partial action, there exists~$\xx$ in~$\XX$ such that $\xx \act \aa$ is
defined, and this  implies that $\xx \act
\aa_\ii$ is defined for each~$\ii$. Next, the hypothesis that $(\aa_1, ..., \aa_\ell)$ is
$\div$-increasing implies that there exist
$\bb_2, ..., \bb_\ell \not= 1$ satisfying $\aa_\ii = \aa_{\ii-1} \bb_\ii$ for each~$\ii$. We find
$$\mu(\xx \act \aa_\ii) = \mu((\xx \act \aa_{\ii-1}) \act \bb_\ii) > \mu(\xx
\act \aa_{\ii-1}),$$ 
and the sequence $(\mu(\xx \act \aa_1), ..., \mu(\xx \act \aa_\ell))$ is increasing. As
$\mu(\xx \act \aa_1) \ge \mu(\xx)$ holds, we deduce $\ell \le \mu(\xx \act \aa) -
\mu(\xx) +1$ and, therefore, $\MM$ satisfies~$\LG$. 
\end{proof}

As for conditions~$\LGi$ and~$\LGii$, we appeal to the subword reversing method
of~\cite{Dgp}. If $\SS$ is any set, we denote by~$\SS^*$ the set of all finite sequences of elements
of~$\SS$, \ie, of words on the alphabet~$\SS$. Then $\SS^*$ equipped with concatenation is a free monoid.
We use~$\ew$ for the empty word.

\begin{defi}
Let $\SS$ be any set. A map $\CC: \SS\times\SS \to \SS^*$ is called a
\emph{complement} on~$\SS$. Then, we denote by~$\RR_\CC$ the
family of all relations $\aa \CC(\aa,\bb) = \bb \CC(\bb,\aa)$ with $\aa \not= \bb$
in~$\SS$, and by~$\CCh$ the unique (possibly partial) map of~$\SS^* \times \SS^*$
to~$\SS^*$ that extends~$\CC$ and obeys the recursive rules
\begin{equation}
\CCh(\uu, \vv_1 \vv_2) = \CCh(\uu,\vv_1) \CCh(\CCh(\vv_1,\uu), \vv_2), \ 
\CCh(\vv_1\vv_2, \uu) =  \CCh(\vv_2, \CCh(\uu,\vv_1)).
\end{equation}
\end{defi}

\begin{prop} [\cite{Dgp} or {\cite[Prop.\,II.2.5.]{Dgd}}]
\label{P:Cube}
Assume that $\MM$ is a monoid satisfying~$\LG$ and admitting the presentation
$\Mon{\SS, \RR_\CC}$, where $\CC$ is a complement on~$\SS$. Then the following are equivalent:

$(i)$ The monoid~$\MM$ is left-preGarside; 

$(ii)$ For all~$\aa, \bb, \cc$ in~$\SS$, we have 
\begin{equation}
\label{E:Cube}
\CCh(\CCh(\CCh(\aa,\bb), \CCh(\aa,\cc)), \CCh(\CCh(\bb, \aa),
\CCh(\bb,\cc))) = \ew.
\end{equation}
\end{prop}

\subsection{Proof of Therem~\ref{T:Main}}

We shall now prove that the monoid~$\MLD$ equip\-ped its partial action on terms via left self-distributivity
satisfies the criteria of Section~\ref{S:Recognizing}. Here, and in most subsequent developments, we heavily
appeal to the results of~\cite{Dgd}, some of which have quite intricate proofs.

\begin{proof}[Proof of Theorem~\ref{T:Main}]
First, each term~$\tt$ has a size~$\mu(\tt)$, which is the number of inner nodes in the associated binary
tree. Then the hypothesis of Proposition~\ref{P:Mass} clearly holds: if $\tt'$
is a nontrivial LD-expansion of~$\tt$, then the size of~$\tt'$ is larger than that
of~$\tt$. Then, by Proposition~\ref{P:Mass},  $\MLD$ satisfies~$\LG$.

Next, we observe that the presentation of~$\MLD$ in
Definition~\ref{D:RLD} is associated with a complement on the set $\{\DD\a \mid \a
\in \Add\}$. Indeed, for each pair of addresses~$\a, \b$, there exists in the list~$\RLD$
exactly one relation of the type $\DD\a ...  = \DD\b ...$. Hence, in view of
Proposition~\ref{P:Cube}, and because we know that $\MLD$ satisfies~$\LG$, it suffices to check that
\eqref{E:Cube} holds in~$\MLD$ for each triple~$\DD\a, \DD\b, \DD\g$. This is Proposition~VIII.1.9
of~\cite{Dgd}. Hence $\MLD$ satisfies~$\LGi$ and~$\LGii$, and it is a left-preGarside monoid.

Let us now consider the elements~$\D_\tt$ of Definition~\ref{D:Delta}. First, by
Lemma~\ref{L:Delta}, $\tt \act \D_\tt$ is defined for each term~$\tt$, and it is equal
to~$\f(\tt)$. Next, assume that $\tt \act \DD\a$ is defined. Then Lemma~VII.3.16
of~\cite{Dgd} states that $\DD\a$ is a left-divisor of~$\D_\tt$ in~$\MLD$,
whereas Lemma~VII.3.17 of~\cite{Dgd} states that $\D_\tt$ is a left-divisor of~$\DD\a
\D_{\tt\act\DD\a}$. Hence Condition~$\LGloc$ of Definition~\ref{D:LocGarside} is satisfied, and the sequence
$(\D_\tt)_{\tt \in \TT}$ is a left-Garside sequence in~$\MLD$. Hence $\MLD$ is a locally left-Garside
monoid, which completes the proof of Proposition~\ref{P:Main}. 

By Proposition~\ref{P:Criterion}, we deduce that $\CLD$, which is $\CAct\MLD\TT$ by definition,  is a
left-Garside category with left-Garside map~$\D$ as defined in~Theorem~\ref{T:Main}.

As for the connection with the braid category~$\BBB$, we saw in Proposition~\ref{P:Proj} that  $\MAct\pi\RH$
is a surjective functor of~$\CLD$ onto~$\BBB$, and it just remains to prove that it preserves right-lcm's. This
follows from the fact that the homomorphism~$\pi$ of~$\MLD$ to~$\BP\infty$ preserves
right-lcm's, which in turn follows from the fact that $\MLD$ and~$\BP\infty$ are associated with
complements~$\CC$ and~$\CCb$ satisfying, for each pair of addresses~$\a, \b$, 
\begin{equation}
\label{E:ProjComp1}
\pi(\CC(\DD\a, \DD\b)) = \CCb(\pi(\DD\a), \pi(\DD\b)).
\end{equation}
Indeed, let $\aa, \bb$ be any two elements of~$\MLD$. Let $\uu, \vv$ be words on the alphabet $\{\DD\a
\mid
\a\in\Add\}$ that represent~$\aa$ and~$\bb$, respectively. By Proposition~II.2.16 of~\cite{Dgd}, the
word~$\CCh(\uu,
\vv)$ exists, and $\uu
\CCh(\uu,\vv)$ represents~$\lcm(\aa, \bb)$. Then $\pi(\uu \CCh(\uu, \vv))$ represents a
common right-multiple of the braids~$\pi(\aa)$ and~$\pi(\bb)$, and, by~\eqref{E:ProjComp1}, we
have
$$\pi(\uu \CCh(\uu, \vv)) = \pi(\uu) \CCbh(\pi(\uu),\pi(\vv)).$$
This shows that the braid represented by~$\pi(\uu \CCh(\uu, \vv))$, which is $\pi(\lcm(\aa,\bb))$ by
definition, is the right-lcm of the braids~$\pi(\aa)$ and~$\pi(\bb)$. So the morphism~$\pi$ preserves
right-lcm's, and the proof of Theorem~\ref{T:Main} is complete.
\end{proof}

\subsection{The Embedding Conjecture}
\label{S:EmbConj}

From the viewpoint of self-distributive algebra, the main benefit of the current approach might be that its
leads to a natural program for possibly establishing the so-called Embedding Conjecture. This conjecture, at the
moment the most puzzling open question involving free LD-systems, can
be stated in several equivalent forms.

\begin{prop} \cite[Section~IX.6]{Dgd}
\label{P:Embedding}
The following are equivalent:\\
$(i)$ The monoid~$\MLD$ embeds in a group;\\
$(ii)$ The monoid~$\MLD$ admits right-cancellation;\\
$(iii)$ The categories $\CLDm$ and $\CLD$ are isomorphic;\\ 
$(iv)$ The functor~$\f$ associated with the category $\CLD$ is injective;\\
$(v)$ For each term~$\tt$, the LD-expansions of~$\tt$ make an upper-semilattice;\\
$(vi)$ The relations of Lemma~\ref{L:RLD} generate all relations that connect the action of~$\DD\a$'s by
self-distributivity.
\end{prop}

Each of the above properties is conjectured to be true: this is the \emph{Embedding Conjecture}.

We turn to the proof of Theorem~\ref{T:Embedding}. So our aim is to show that the Embedding Conjecture is
true whenever the category~$\CLD$ is regular. To this end, we shall use some technical results
from~\cite{Dgd}, plus the following criterion, which enables one to prove right-cancellability by only using
simple morphisms.

\begin{prop}
\label{P:RightCancellation}
Assume that $\CCC$ is a left-Garside category and the associated functor~$\f$ is injective on~$\Obj(\CCC)$.
Then the following are equivalent: 

$(i)$ $\Hom(\CCC)$ admits right-cancellation;

$(ii)$ The functor~$\f$ is injective on~$\Hom(\CCC)$.\\
Moreover, if $\CCC$ is regular, $(i) $ and $(ii)$ are equivalent to

$(iii)$ The functor~$\f$ is injective on simple morphisms of~$\CCC$.
\end{prop}

\begin{proof}
Assume that $\ff, \gg$ are morphisms of~$\CCC$ that satisfy $\f(\ff) = \f(\gg)$. As $\f$ is a functor,
we first deduce
$$\f(\source\ff) = \source(\f(\ff)) = \source(\f(\gg)) = \f(\source\gg),$$
hence $\source\ff = \source\gg$ as $\f$ is injective on objects. A similar argument gives
$\target\ff = \target\gg$. Then, \eqref{E:Phi} gives
$$\ff \D(\target\ff) = \D(\source\ff) \f(\ff) = \D(\source\gg) \f(\gg) 
= \gg \D(\target\gg) = \gg\D(\target\ff).$$
If we can cancel $\D(\target\ff)$ on the right, we deduce $\ff =\gg$ and, therefore, $(i)$ implies~$(ii)$. 

Conversely, assume that $\hh$ is simple and $\ff \hh = \gg \hh$ holds. By multiplying
by~$\hh^*$, we deduce
$\ff \hh \hh^*  = \gg \hh \hh^*$, \ie, $\ff \D(\source\hh) = \gg \D(\source\hh)$.
As we have $\target\ff = \source\hh = \target\gg$ by hypothesis, 
applying~\eqref{E:Phi} gives
$$\D(\source\ff) \f(\ff) = \ff \D(\target\ff) = \gg \D(\target\gg) = \D(\source\gg)
\f(\gg) = \D(\source\ff) \f(\gg),$$
hence $\f(\ff) = \f(\gg)$ by left-cancelling~$\D(\source\ff)$. If $(ii)$ holds, we
deduce~$\ff = \gg$, \ie, $\hh$ is right-cancellable. As simple morphisms generate~$\Hom(\CCC)$, we deduce
that every morphism is right-cancellable and, therefore, $(ii)$ implies~$(i)$.

It is clear that $(ii)$ implies~$(iii)$. So assume that $\CCC$ is regular and $(iii)$ holds. Let $\ff, \gg$ satisfy
$\f(\ff) = \f(\gg)$. Let $(\ff_1, ..., \ff_\dd)$ and $(\gg_1, ..., \gg_\ee)$ be the normal forms of~$\ff$
and~$\gg$, respectively. The regularity assumption implies that every length~$2$ subsequence of $(\f(\ff_1),
..., \f(\ff_\dd))$ and $(\f(\gg_1), ..., \f(\gg_\ee))$ is normal. Moreover, $(iii)$ guarantees that $\f(\ff_\ii)$
and $\f(\gg_\jj)$ is nontrivial. Hence $(\f(\ff_1), ..., \f(\ff_\dd))$ and $(\f(\gg_1), ..., \f(\gg_\ee))$ are
normal. As $\f$ is a functor, we have $\f(\ff_1) ... \f(\ff_\dd) = \f(\ff) = \f(\gg) = \f(\gg_1) ...
\f(\gg_\ee)$, and the uniqueness of the normal form implies $\dd = \ee$, and $\f(\ff_\ii) = \f(\gg_\ii)$ for
each~$\ii$. Then $(iii)$ implies $\ff_\ii = \gg_\ii$ for each~$\ii$, hence $\ff = \gg$.
\end{proof}

So, in order to prove Theorem~\ref{T:Embedding}, it suffices to show that the category~$\CLD$ satisfies the
hypotheses of Proposition~\ref{P:RightCancellation}, and this is what we do now.

\begin{lemm}
\label{L:PartialInj}
The functor~$\f$ of~$\CLD$ is injective on objects, \ie, on terms.
\end{lemm}

\begin{proof}
We show using induction on the size of~$\tt$ that $\f(\tt)$
determines~$\tt$. The result is obvious if $\tt$ has size~$0$. Assume
$\tt = \tt_\0 \op \tt_\1$. By construction, the term~$\f(\tt)$ is obtained by
substituting every variable~$\xx_\ii$ occurring in the term~$\f(\tt_\1)$ with the
term~$\f(\tt_\0)\op\xx_\ii$. Hence $\f(\tt_\0)$ is the $\1^{\nn-1}\0$th subterm
of~$\f(\tt)$, where $\nn$ is the common right-height of~$\tt$ and~$\f(\tt)$. From
there, $\f(\tt_\1)$ can be recovered by replacing the subterms~$\f(\tt_\0)\op\xx_\ii$
of~$\f(\tt)$ by~$\xx_\ii$. Then, by induction hypothesis, $\tt_\0$ and~$\tt_\1$,
hence~$\tt$, can be recovered from~$\f(\tt_\1)$ and~$\f(\tt_\0)$.
\end{proof}

\begin{lemm}
\label{L:PartialInj2}
The functor~$\f$ of~$\CLD$ is injective on simple morphisms.
\end{lemm}

\begin{proof}
Assume that $\ff, \ff'$ are morphisms of~$\CLD$ satisfying 
$\f(\ff) = \f(\ff')$, say \linebreak
 $\ff = (\tt, \aa, \ss)$ and $\ff' = (\tt', \aa', \ss')$.
The explicit description of Lemma~\ref{L:Phi} implies $\f(\tt) = \f(\tt')$, hence $\tt =
\tt'$ by Lemma~\ref{L:PartialInj}. Similarly, we have $\f(\ss) = \f(\ss')$, hence $\ss' =
\ss$. Therefore, we have $\tt \act \aa = \tt \act \aa' = \ss$. By Proposition~VII.126
of~\cite{Dgd}, we deduce that $\tt \act \aa = \tt \act \aa'$ holds for every term~$\tt$ for
which both $\tt \act \aa$ and $\tt \act \aa'$ are defined. Then Proposition~IX.6.6
of~\cite{Dgd} implies $\aa = \aa'$ provided $\aa$ or~$\aa'$ is simple.
\end{proof}

We can now complete the argument.

\begin{proof}[Proof of Theorem~\ref{T:Embedding}]
The category~$\CLD$ is left-Garside, with an associated functor~$\f$ that is injective both on objects and on
simple morphisms. By Proposition~\ref{P:RightCancellation}, if $\CLD$ is regular, then $\Hom(\CLD)$
admits right-cancellation, which is one of the forms of the Embedding Conjecture, namely $(ii)$
in Proposition~\ref{P:Embedding}.
\end{proof}

\subsection{A program for proving the regularity of~$\CLD$}

At this point, we are left with the question of proving (or disproving)

\begin{conj}
\label{C:Main} 
The left-Garside category $\CLD$ is regular.
\end{conj}

The regularity criteria of Section~\ref{S:Criteria} lead to a natural program for possibly proving
Conjecture~\ref{C:Main} and, therefore, the Embedding Conjecture. 

We begin with a preliminary observation.

\begin{lemm}
\label{L:Coherent}
The left-Garside sequence $(\D_\tt)_{\tt \in \TT}$ on~$\MLD$ is coherent (in the sense of
Definition~\ref{D:Coherent}).
\end{lemm}

\begin{proof}
The question is to prove that, if $\tt$ is a term and $\tt \act \aa$ is defined and $\aa \dive \D_{\tt'}$ holds
for some~$\tt'$, then we necessarily have $\aa \dive \D_\tt$. This is a direct consequence of
Proposition~VIII.5.1 of~\cite{Dgd}. Indeed, the latter states that an element~$\aa$ is a left-divisor of some
element~$\D_\tt$ if and only if $\aa$ can be represented by a word in the letters~$\DD\a$
that has a certain special form. This property does not involve the term~$\tt$, and it implies that, if $\aa$
left-divides~$\D_\tt$, then it automatically left-divides every element~$\D_{\tt'}$ such that $\tt' \act \aa$ is
defined.
\end{proof}

So, according to Proposition~\ref{P:Coherent}, we obtain a well defined notion of a simple element
in~$\MLD$: an element~$\aa$ of~$\MLD$ is called \emph{simple} if it left-divides at least one element of the
form~$\D_\tt$. Then simple elements form a seed in~$\MLD$, and are eligible for a normal form satisfying
the general properties described in Section~\ref{S:Normal}.
In this context, applying Proposition~\ref{P:CriterionTer}$(ii)$ leads to the following criterion.

\begin{prop}
\label{P:PhiGcd}
Assume that, for each term~$\tt$ and all simple elements~$\aa, \bb$ of~$\MLD$ such that $\tt
\act \aa$ and $\tt \act \bb$ are defined, we have
\begin{equation}
\label{E:Gcd}
\gcd(\f_\tt(\aa), \f_\tt(\bb)) = \f_\tt(\gcd(\aa, \bb)).
\end{equation}
Then Conjecture~\ref{C:Main} is true.
\end{prop}

\begin{proof}
Let $\ff, \gg$ be two simple morphisms in~$\CLD$ that satisfy $\source\ff = \source\gg = \tt$. By definition,
$\ff$ has the form $(\tt, \aa, \tt\act\aa)$ for some~$\aa$ satisfying $\aa \dive \D_{\tt}$, hence simple
in~$\MLD$. Similarly, $\ff$ has the form $(\tt, \bb, \tt\act\bb)$ for some simple element~$\bb$, and
we have $\gcd(\ff, \gg) = (\tt, \gcd(\aa, \bb), \tt\act \gcd(\aa, \bb))$. On the other hand, 
Lemma~\ref{L:Phi} gives $\f(\ff) = (\f(\tt), \f_\tt(\aa), \f(\tt \act \aa))$ and. 
$\f(\gg) = (\f(\tt), \f_\tt(\bb), \f(\tt \act \bb))$, whence
$$\gcd(\f(\ff), \f(\gg)) = (\f(\tt), \gcd(\f_\tt(\aa),\f_\tt(\bb)), \f(\tt) \act \gcd(\f_\tt(\aa),\f_\tt(\bb))).$$
If \eqref{E:Gcd} holds, we deduce 
$$\gcd(\f(\ff), \f(\gg)) = \f(\gcd(\ff, \gg)).$$
By Proposition~\ref{P:CriterionTer}$(ii)$, this implies that $\CLD$ is regular.
\end{proof}

\begin{exam}
Assume $\aa = \DD\ea$, $\bb = \DD1$, and $\tt =\xx\op(\xx\op(\xx\op\xx))$. Then $\tt \act \aa$ and $\tt
\act \bb$ are defined. On the other hand, we have $\f(\tt) =
((\xx\op\xx)\op(\xx\op\xx))\op((\xx\op\xx)\op(\xx\op\xx))$. An easy computation gives $\f_\tt(\DD\ea) =
\DD0\DD1$ and $\f_\tt(\DD1) = \DD\ea$, see Figure~\ref{F:Gcd}. We find $\gcd(\f_\tt(\aa), \f_\tt(\bb)) = 1
= \gcd(\aa, \bb)$, and
\eqref{E:Gcd} is true in this case.

Note that the couterpart of~\eqref{E:Gcd} involving right-lcm's fails. In the current case, we have
$$\lcm(\f_\tt(\aa), \f_\tt(\bb)) = \f_\tt(\lcm(\aa, \bb)) \cdot \DD0\DD1:$$
the terms $\f(\tt \act \DD\ea)$ and $\f(\tt \act \DD1)$ admit a common LD-expansion that is smaller than
$\f_\tt(\tt \act \lcm(\DD\ea, \DD1))$, which turns out to be~$\f^2(\tt)$, see
Figure~\ref{F:Gcd} again. 

The reader may similarly check that \eqref{E:Gcd} holds for $\tt =
(\xx \op (\xx \op \xx)) \op (\xx \op (\xx \op \xx))$ with $\aa = \DD0$ and $\bb = \DD1$; the values are
$\f_\tt(\DD0) =
\DD{000}\DD{010}\DD{100}\DD{110}$ and $\f_\tt(\DD1) = \DD\ea$.
\end{exam}

\begin{figure}[tb]
\begin{picture}(120,63)(0,0)
\put(0,0){\includegraphics{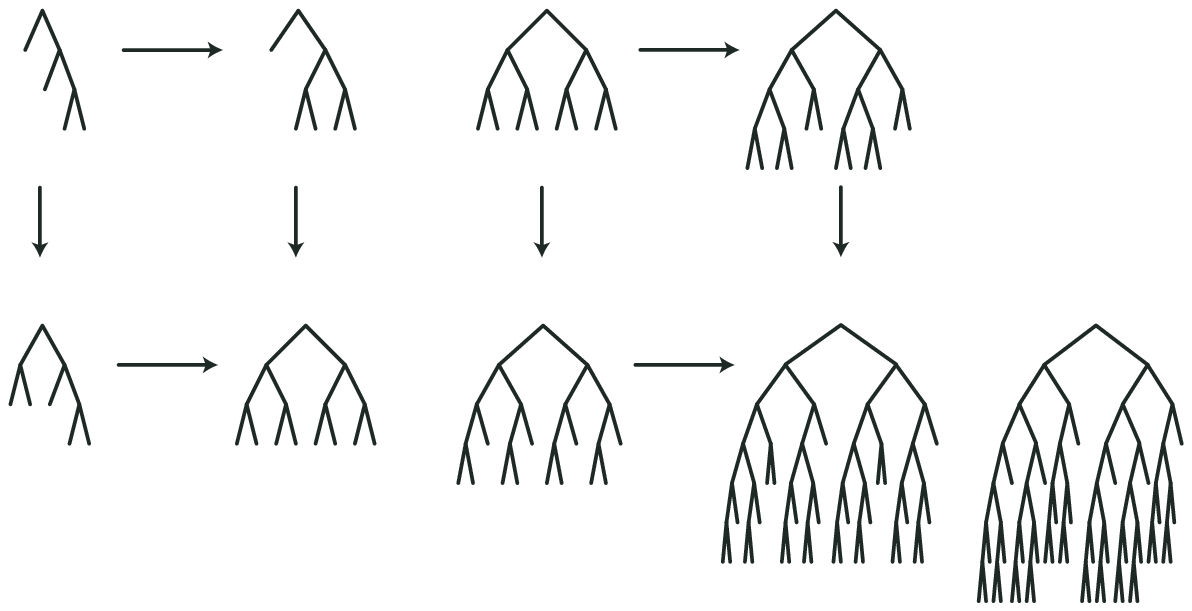}}
\put(3,62){$\tt$}
\put(26,62){$\tt\act\DD1$}
\put(52,62){$\f(\tt)$}
\put(78,62){$\f(\tt\act\DD1)$}
\put(14,58){$\DD1$}
\put(67,58){$\DD\ea$}
\put(5,38){$\DD\ea$}
\put(30,38){$\DD\ea\DD0\DD1$}
\put(55,38){$\DD0\DD1$}
\put(86,38){$\DD{00}\DD{10}\DD1\DD\ea$}
\put(0,30){$\tt\act\DD\ea$}
\put(20,30){$\tt\!\act\!\lcm(\!\DD\ea,\!\DD1\!)$}
\put(48,30){$\f(\tt\!\act\!\DD\ea)$}
\put(70,30){$\f(\tt)\!\act\!\lcm(\!\DD0\!\DD1,\!\DD\ea\!)$}
\put(100,30){$\f(\tt\!\act\!\lcm(\!\DD\ea,\!\DD1\!))$}
\put(11,26){$\DD1\DD\ea$}
\put(62,26){$\DD\ea\DD1\DD\ea$}
\put(97,24){$\not=$}
\end{picture}
\caption{\sf\smaller  The left diagram shows an instance of Relation~\eqref{E:Gcd}: for the considered choice
of~$\tt$, we find $\D_\tt = \DD\ea\DD1\DD\ea$, 
$\D_{\tt \act \DD1} = \DD1\DD\ea\DD0\DD{00}\DD1\DD{10}$, leading to  $\f_\tt(\DD1) = \DD\ea$
and $\f_\tt(\DD\ea) = \DD0\DD1$. Here $\f_\tt(\DD\ea)$ and
$\f_\tt(\DD1)$ are left-coprime, so~\eqref{E:Gcd} is true. The right diagram shows that the counterpart
involving lcm's fails.}
\label{F:Gcd}
\end{figure}

Proposition~\ref{P:PhiGcd} leads to a realistic program that would reduce the proof of the
Embedding Conjecture to a (long) sequence of verifications. Indeed, it is shown
in Proposition~VIII.5.15 of~\cite{Dgd} that every simple element~$\aa$ of~$\MLD$ admits a
unique expression of the form
$$\aa = \prod_{\a \in \Add}^{>} \DD\a^{(\ee_\a)},$$
where $\DD\a^{(\ee)}$ denotes $\DD\a \DD{\a\1} ... \DD{\a\1^{\ee-1}}$ and $>$ refers
to the unique linear ordering of~$\Add$ satisfying $\a > \a\0\b > \a\1\g$ for all~$\a,
\b, \g$. In this way, we associate with every simple element~$\aa$ of~$\MLD$ a sequence
of nonnegative integers $(\ee_\a)_{\a \in \Add}$ that plays the role of a sequence of coordinates
for~$\aa$. Then it \emph{should} be possible to 

- express the coordinates of~$\f_\tt(\aa)$ in terms of those of~$\aa$,

- express the coordinates of~$\gcd(\aa, \bb)$ in terms of those of~$\aa$ and~$\bb$.\\
If this were done, proving (or disproving) the equalities~\eqref{E:Gcd} should be easy.

\begin{rema*}
Contrary to the braid relations, the LD-relations of Lemma~\ref{L:RLD} are not symmetric. However, it turns
out that the presentation of~$\MLD$ is also associated with what can naturally be called a left-complement,
namely a counterpart of a (right)-complement involving left-multiples. But the latter fails to satisfy the
counterpart of~\eqref{E:Cube}, and it is extremely unlikely that one can prove that the monoid~$\MLD$ is
possibly right-cancellative (which would imply the Embedding Conjecture) using some version of
Proposition~\ref{P:Cube}.
\end{rema*}

\section{Reproving  braids properties}
\label{S:Reproving}

Proposition~\ref{P:Proj} and Theorem~\ref{T:Main} connect the Garside structures associated with
self-distributi\-vity and with braids, both being previously known to exist. In this
section, we show how the existence of the Garside structure of braids can be (re)-proved to exist assuming the
existencce of the Garside structure of~$\CLD$ only. So, for a while, we pretend that we do not know that the
braid monoid~$\BP\nn$ has a Garside structure, and we only know about the Garside structure
of~$\CLD$.

\subsection{Projections}
\label{S:Projection}

We begin with a general criterion guaranteeing that the projection of a locally left-Garside monoid is again a
locally left-Garside monoid. 

If $\pi$ is a map of~$\SS$ to~$\SSb^*$, we still denote by~$\pi$ the alphabetical homomorphism of~$\SS^*$
to~$\SSb^*$ that extends~$\pi$, defined by~$\pi(\ss_1 ... \ss_\ell) = \pi(\ss_1)...\pi(\ss_\ell)$.

\begin{lemm}
\label{L:Homo}
Assume that 

$\bullet$ $\MM$ is a locally left-Garside monoid associated with a complement~$\CC$ on~$\SS$;

\smallskip
$\bullet$ $\MMb$ is a monoid associated with a complement~$\CCb$ on~$\SSb$ and
satisfying~$\LG$;

\smallskip
$\bullet$ $\pi: \SS \to \SSb \cup \{\ew\}$ satisfies $\pi(\SS) \supseteq \SSb$  and
\begin{equation}
\label{E:ProjDef}
\mbox{For all~$\aa, \bb$ in~$\SS$, we have $\CCb(\pi(\aa), \pi(\bb)) = \pi(\CC(\aa,\bb))$.}
\end{equation}
\vskip-1mm
Then $\MMb$ is left-preGarside, and $\pi$ induces a surjective right-lcm preserving
homomorphism of~$\MM$ onto~$\MMb$.
\end{lemm}

\begin{proof}
An easy induction shows that, if $\uu, \vv$ are words on~$\SS$ and $\CCh(\uu,\vv)$
exists, then $\CCbh(\pi(\uu), \pi(\vv))$ exists as well and we have
\begin{equation}
\label{E:ProjComp}
\CCbh(\pi(\uu), \pi(\vv)) = \pi(\CCh(\uu,\vv)).
\end{equation}

Let $\aab, \bbb, \ccb$ be elements of~$\SSb$. By hypothesis, there exist $\aa, \bb, \cc$
in~$\SS$ satisfying $\pi(\aa) = \aab$, $\pi(\bb) = \bbb$, $\pi(\cc) = \ccb$. As $\MM$ is left-preGarside, by
the direct implication of Proposition~\ref{P:Cube}, the relation~\eqref{E:Cube} involving~$\aa, \bb, \cc$ is true
in~$\SS^*$. Applying~$\pi$ and using~\eqref{E:ProjComp}, we deduce that the
relation~\eqref{E:Cube} involving~$\aab, \bbb, \ccb$ is true in~$\SSb^*$. Then, as $\MMb$
satisfies~$\LG$ by hypothesis, the converse implication of Proposition~\ref{P:Cube} implies that $\MMb$ is
left-preGarside.

Then, by definition, the relations $\aa \CC(\aa, \bb) = \bb \CC(\bb, \aa)$ with $\aa, \bb \in \SS$ make a
presentation of~$\MM$. Now, for~$\aa, \bb$ in~$\SS$, we find
$$\pi(\aa) \CCb(\pi(\aa), \pi(\bb)) 
= \pi(\aa \CC(\aa, \bb)) 
= \pi(\bb \CC(\bb, \aa)) 
= \pi(\bb) \CC(\pi(\bb), \pi(\aa))$$ in~$\MMb$, which shows that the
homomorphism of~$\SS^*$ to~$\MMb$ that extends~$\pi$ induces a well defined homomorphism of~$\MM$
to~$\MMb$. This homomorphism, still denoted~$\pi$, is surjective since, by hypothesis, its image
includes~$\SSb$.

Finally, we claim that $\pi$ preserves right-lcm's. The argument is almost the same as in the proof of
Theorem~\ref{T:Main}, with the difference that, here, we do not assume that common multiples necessarily
exist. Let $\aa, \bb$ be two elements of~$\MM$ that admit a common right-multiple. Let
$\uu, \vv$ be words on~$\SS^*$ that represent~$\aa$ and~$\bb$, respectively. By Proposition~II.2.16
of~\cite{Dgd}, the word~$\CCh(\uu,
\vv)$ exists, and $\uu
\CCh(\uu,\vv)$ represents~$\lcm(\aa, \bb)$. Then the word $\pi(\uu \CCh(\uu, \vv))$ represents a
common right-multiple of~$\pi(\aa)$ and~$\pi(\bb)$ in~$\MMb$, and, by~\eqref{E:ProjComp}, we have
$$\pi(\uu \CCh(\uu, \vv)) = \pi(\uu) \; \CCbh(\pi(\uu),\pi(\vv)),$$
which shows that the element represented by~$\pi(\uu \CCh(\uu, \vv))$, which is $\pi(\lcm(\aa,\bb))$ by
definition, is the right-lcm of~$\pi(\aa)$ and~$\pi(\bb)$ in~$\MMb$.
\end{proof}

We turn to locally left-Garside monoids, \ie, we add partial actions in the picture. Although
lengthy, the following result is easy. It just says that, if $\MM$ is a locally left-Garside
monoid, then its image under a projection that is compatible with the various ingredients of the
Garside structure is again locally left-Garside.

\begin{prop}
\label{P:Projection}
Assume that 

$\bullet$ $\MM$ is a locally left-Garside monoid associated with a complement~$\CC$ on~$\SS$
and $(\D_\xx)_{\xx\in\XX}$ is a left-Garside sequence for the involved action of~$\MM$ on~$\XX$;

\smallskip
$\bullet$ $\MMb$ is a monoid associated with a complement~$\CCb$ on~$\SSb$ that has a partial
action on~$\XXb$ and satisfies~$\LG$;

\smallskip
$\bullet$ $\pi: \SS \to \SSb \cup \{\ew\}$ satisfies \eqref{E:ProjComp}, $\theta: \SSb \to \SS$ is a
section for~$\pi$, $\pio: \XX \to \XXb$ is a surjection, and\\
\parbox{\hsize}{
\begin{equation}
\label{E:Proj4}
\parbox{11cm}{%
\noindent For~$\xx$ in~$\XX$ and~$\aa$ in~$\MM$, if $\xx\act\aa$ is defined, then so is $\pio(\xx) \act
\pi(\aa)$\\
\null\hfill\strut and we have $\pio(\xx) \act \pi(\aa) = \pio(\xx\act\aa)$;}
\end{equation}
\vskip-4mm
\begin{equation}
\label{E:Proj5}
\parbox{11cm}{%
\noindent For~$\xxb$ in~$\XXb$ and~$\aab$ in~$\SSb$, if $\xxb\act\aab$ is defined, \\
\null\hfill\strut then so is $\xx \act \theta(\aab)$ for each~$\xx$ satisfying  $\pio(\xx)  = \xx$;}
\end{equation}
\vskip-4mm
\begin{equation}
\label{E:Proj6}
\parbox{11cm}{%
\noindent For~$\xx$ in~$\XX$, the value of~$\pi(\D_\xx)$ depends on~$\pio(\xx)$ only.}
\end{equation}}
\vskip-1mm
For~$\xxb$ in~$\XXb$, let $\D_\xxb$ be the common value of~$\pi(\D_\xx)$ for $\pio(\xx) = \xxb$. Then
$\MMb$ is locally left-Garside, with associated left-Garside sequence
$(\D_\xxb)_{\xxb\in\XXb}$, and $\pi$
induces a surjective right-lcm preserving homomorphism of~$\MM$ onto~$\MMb$.
\end{prop}

\begin{proof}
First, the hypotheses of Lemma~\ref{L:Homo} are satisfied, hence $\MMb$ is left-preGarside and $\pi$
induces a surjective lcm-preserving homomorphism of~$\MM$ onto~$\MMb$.

Next, by~\eqref{E:Proj6}, the definition of the elements~$\D_\xxb$ for $\xxb$ in~$\XXb$ is unambiguous. It
remains to check that the sequence $(\D_\xxb)_{\xxb\in\XXb}$ is a left-Garside sequence with respect to the
action of~$\MMb$ on~$\XXb$. So, assume $\xxb \in \MMb$, and let $\xx$ be any element of~$\MM$
satisfying $\pio(\xx) = \xxb$. 

First,  $\xx \act \D_\xx$ is defined, hence, by~\eqref{E:Proj4}, so is $\pio(\xx) \act \pi(\D_\xx)$, which is
$\xxb \act \D_\xxb$.

Assume now $\aab \not= 1$ and $\xxb \act \aab$ is defined. As $\SSb$ generates~$\MMb$, we
can assume $\aa \in \SSb$ without loss of generality. By~\eqref{E:Proj5}, the existence of~$\xxb
\act \aab$ implies that of $\xx \act \theta(\aab)$. As $(\D_\xx)_{\xx \in \XX}$ is a left-Garside sequence for
the action of~$\MM$ on~$\XX$, we have $\aa' \dive \D_\xx$ for some $\aa' \not= 1$
left-dividing~$\theta(\aab)$. By construction, $\theta(\aab)$ lies in~$\SSb$, and it is an
atom in~$\MM$. So the only possiblity is $\aa' = \theta(\aab)$, \ie, we have
$\theta(\aab) \dive
\D_\xx$. Applying~$\pi$, we deduce $\aab \dive \D_\xxb$ in~$\MMb$.

Finally, under the same hypotheses, we have $\D_\xx \dive \theta(\aab) \, \D_{\xx \act \theta(\aab)}$
in~$\MM$. Using
$$\pi(\D_{\xx \act \theta(\aab)}) = \D_{\pio(\xx \act \theta(\aab))}
= \D_{\pio(\xx) \act \aab} = \D_{\xxb \act \aab},$$
we deduce $\D_\xxb \dive \aab \, \D_{\xxb \act \aab}$ in~$\MMb$, always under the hypothesis $\aab \in
\SSb$. The case of an arbitrary element~$\aab$ for which $\xxb \act \aab$ exists then follows from an easy
induction on the length of an expression of~$\aab$ as a product of elements of~$\SSb$.
\end{proof}

It should then be clear that, under the hypotheses of Proposition~\ref{P:Projection}, $\MAct\pi\pio$ is
a surjective, right-lcm preserving functor of~$\CAct\MM\XX$ to~$\CAct\MMb\XXb$.

\subsection{The case of~$\CLD$ and~$\BBB$}

Applying the criterion of Section~\ref{S:Projection} to the categories~$\CLD$ and~$\BBB$ is easy. 

\begin{prop}
The monoid~$\BP\infty$ is a locally left-Garside monoid with respect to its action on~$\Nat$, and
$(\D_\nn)_{\nn\in\Nat}$ is a left-Garside sequence in~$\BP\infty$. 
\end{prop}

\begin{proof}
Hereafter, we denote by~$\CC$ the complement on $\{\DD\a\mid\a \in \Add\}$ associated with the
LD-relations of Lemma~\ref{L:RLD}, and by~$\CCb$ the complement on $\{\sig\ii \mid \ii \ge 1\}$ associated
with the braid relations of~\eqref{E:Braid}. We consider the maps~$\pi$ of Lemma~\ref{L:Proj}, 
and~$\RH$ from terms to nonnegative integers. Finally, we define $\theta$ by $\theta(\sig\ii) =
\DD{\1^{\ii-i}}$. We claim that these data satisfy all hypotheses of
Proposition~\ref{P:Projection}. The verifications are easy. That the complements~$\CC$ and~$\CCb$
satisfy~\eqref{E:ProjDef} follows from a direct inspection. For instance, we find
$$\pi(\CC(\DD1, \DD\ea)) = \pi(\DD\ea \DD1 \DD0) = \sig1 \sig2 = \CCb(\sig2, \sig1) = \CCb(\pi(\DD1),
\pi(\DD\ea)),$$
and similar relations hold for all pairs of generators~$\DD\a, \DD\b$. 

Then, the action of~$\MLD$ on terms preserve the right-height, whereas the action of braids on~$\Nat$ is
trivial, so \eqref{E:Proj4} is clear. Next, define $\theta$ by $\theta(\sig\ii) = \DD{\1^{\ii-1}}$. Then
$\theta$ is a section for~$\pi$, and we observe that $\tt \act \theta(\sig\ii)$ is defined if and only if the
right-height of~$\tt$ is at least~$\ii+1$, hence if and only if $\RH(\tt) \act \sig\ii$ is defined, so
\eqref{E:Proj5} is satisfied. Finally, we observed in Proposition~\ref{P:Proj} that $\pi(\D_\tt)$ is equal
to~$\D_{\RH(\tt)}$, hence it depends on~$\RH(\tt)$ only. So \eqref{E:Proj6} is satisfied.

Therefore, Proposition~\ref{P:Projection} applies, and it gives the expected result.
\end{proof}

\begin{coro}
$(i)$ The braid category~$\BBB$ is a left-Garside category.

$(ii)$ For each~$\nn$, the braid monoid~$\BP\nn$ is a Garside monoid.
\end{coro}

\begin{proof}
Point~$(i)$ follows from Proposition~\ref{P:Criterion} once we know that $\BP\infty$ is locally
left-Garside. Point~$(ii)$ follows from Proposition~\ref{P:Submonoid} since, for each~$\nn$, the
submonoid~$\BP\nn$ is the monoid~$(\BP\infty)_\nn$ in the sense of Definition~\ref{D:Action}.
\end{proof}

Thus, as announced, the Garside structure of braids can be recovered from the left-Garside structure associated
with the LD-law. 

\section{Intermediate categories}
\label{S:Intermediate}

We conclude with a different topic. The projection of the self-distributivity category $\CLD$ to the braid
category~$\BBB$ described above is partly trivial in that terms are involved through their right-height only
and the corresponding action of braids on integers is just constant. Actually,
one can consider alternative projections corresponding to less trivial braid actions and leading to two-step
projections
$$\CLD \ \longrightarrow\  \CAct{\BP\infty}\XX \ \longrightarrow\  \BBB.$$
We shall describe two such examples. 

\subsection{Action of braids on sequences of integers}

Braids act on sequences of integers via their permutations. Indeed, the rule
\begin{equation}
\label{E:PermAction}
(\xx_1, ..., \xx_\nn) \act \sig\ii = (\xx_1, ..., \xx_{\ii-1}, \xx_{\ii+1}, \xx_\ii, \xx_{\ii+2} , ...,
\xx_\nn).
\end{equation}
defines an action of~$\BP\nn$ on~$\Nat^\nn$, whence a partial action of~$\BP\infty$ on~$\Nat^*$, where
$\Nat^*$ denotes the set of all finite sequences in~$\Nat$. In this way, we obtain
a new category~$\CAct{\BP\infty}{\Nat^*}$, which clearly projects to~$\BBB$. 

We shall now describe an explicit projection of~$\CLD$ onto this category. We recall that terms have been
defined to be bracketed expressions constructed from a fixed sequence of variables~$\xx_1, \xx_2, ...$ (or as
binary trees with leaves labelled with variables~$\xx_\pp$), and that, for~$\tt$ a term and $\a$ a binary
address, $\sub\tt\a$ denotes the subterm of~$\tt$ whose root, when $\tt$ is viewed as a binary tree, has
address~$\a$. 

\begin{prop}
\label{P:BraidHat}
Let $\BBBB$ be the category associated with the partial action~\eqref{E:PermAction} of~$\BP\infty$
on~$\Nat^*$. Then $\BBBB$ is a Garside category, and the projection~$[\pi, \RH]$ of~$\CLD$ onto~$\BBB$
factors through~$\BBBB$ into
$$ 
\begin{CD}
\CLD
@>{\textstyle[\pi, \PI]}>>
\BBBB
@>{\textstyle[\ID, \lg]}>>
\BBB,
\end{CD}
$$
where $\PI$ is defined for $\RH(\tt) = \nn$ by
$$
\PI(\tt) = (\var(\sub{\tt}{\0}),\var(\sub{\tt}{\1\0}), ..., \var(\sub{\tt}{\1^\nno\0})),
$$
$\var(\tt)$ denoting the index of the righmost variable occurring in~$\tt$. 
\end{prop}

So, a typical morphism of~$\BBBB$ is $((1,2,2), \sig1, (2,1,2))$, and the projection of terms to sequences of
integers is given by 
\vrule width0pt height14mm depth15mm
\begin{picture}(72,10)(4,12)
\put(6,0.5){\includegraphics{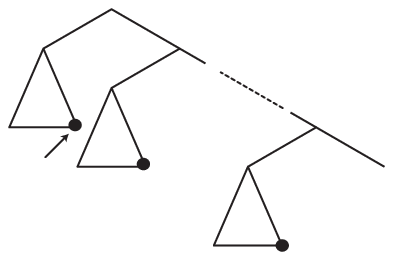}}
\put(7,8){$\xx_{\pp_1}$}
\put(18,6.5){$\xx_{\pp_2}$}
\put(32,-1.5){$\xx_{\pp_\nn}$}
\put(45,12){$\mapsto (\pp_1, \pp_2, ..., \pp_\nn)$.}
\put(46,14){$\PI$}
\end{picture}

\begin{proof}[Proof (Sketch)]
The point is to check that the action of the LD-law on the indices of the right variables of the
subterms with addresses~$\1^{\ii}\0$ is compatible with the action of braids on sequences of
integers. It suffices to consider the basic case of~$\DD{1^{\ii-1}}$, and
the expected relation is shown in Figure~\ref{F:ProjHatProof}. Details are easy. Note that, for symmetry reasons,
the category~$\BBBB$ is not only left-Garside, but even Garside.
\end{proof}

\begin{figure}[tb]
\begin{picture}(48,31)(0,2)
\put(0,0.5){\includegraphics{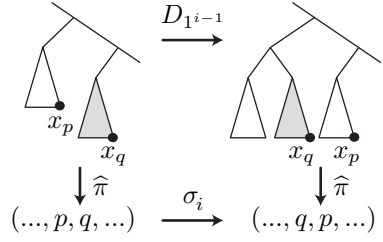}}
\put(10,5){$\PI$}
\put(42,5){$\PI$}
\put(4,14){$\xx_\pp$}
\put(11,10){$\xx_\qq$}
\put(42,10){$\xx_\pp$}
\put(36,10){$\xx_\qq$}
\put(-1,1){$(..., \pp, \qq, ...)$}
\put(31,1){$(..., \qq, \pp, ...)$}
\put(22,4){$\sig\ii$}
\put(19,28){$\DD{\1^{\ii-1}}$}
\end{picture}
\caption{\sf\smaller  Compatibility of the action of~$\DD{\1^{\ii-1}}$ on sequences
of ``subright'' variables and of the action of~$\sig\ii$ on sequences of integers.}
\label{F:ProjHatProof}
\end{figure}

\subsection{Action of braids on LD-systems}

The action of positive braids on sequences of integers defined in~\eqref{E:PermAction} is just one example of a
much more general situation, namely the action of positive braids on sequences of elements of any
LD-system. It is well known---see, for instance,~\cite[Chapter~I]{Dgd}---that, if $(\SS, \op)$ is an LD-system,
\ie,
$\op$ is a binary operation on~$\SS$ that obeys the LD-law, then
\begin{equation}
\label{E:LDAction}
(\xx_1, ..., \xx_\nn) \act \sig\ii = (\xx_1, ..., \xx_{\ii-1}, \xx_\ii \op \xx_{\ii+1}, \xx_\ii, \xx_{\ii+2} , ...,
\xx_\nn)
\end{equation}
induces a well defined action of the monoid~$\BP\nn$ on the set~$\SS^\nn$, and, from there, a
partial action of~$\BP\infty$ on the set~$\SS^*$ of all finite sequences of elements of~$\SS$. 

\begin{prop}
Assume that $(\SS, \op)$ is an LD-system, and let~$\BBB_\SS$ be the category associated with the
partial action~\eqref{E:LDAction} of~$\BP\infty$ on~$\SS^*$. Then $\BBB_\SS$ is a left-Garside category, and,
for all~$\ss_1, \ss_2$, ... in~$\SS$, the projection~$[\pi, \RH]$ of~$\CLD$ onto~$\BBB$ factors
through~$\BBB_\SS$ into
$$ 
\begin{CD}
\CLD
@>{\textstyle\ [\pi, \pi_\SS]\ }>>
\BBB_\SS
@>{\textstyle\ [\ID, \lg]\ }>>
\BBB,
\end{CD}
$$
where $\pi_\SS$ is defined for $\RH(\tt) = \nn$ by
$$\pi_\SS(\tt) = (\eval_\SS(\sub\tt\0), ..., \eval_\SS(\sub\tt{\1^{\nn-1}\0})),$$
$\eval_\SS(\tt)$ being the evaluation of~$\tt$ in~$(\SS, \op)$ when $\xx_\pp$ is given the
value~$\ss_\pp$ for each~$\pp$.
\end{prop}

We skip the proof, which is an easy verification similar to that of Proposition~\ref{P:BraidHat}.

When $(\SS, \op)$ is $\Nat$ equipped with $\xx \op \yy = \yy$ and we map $\xx_\pp$ to~$\pp$, we
obtain the category~$\BBBB$ of Proposition~\ref{P:BraidHat}. In this case, the (partial) action of braids is
not constant as in the case of~$\BBB$, but it factors through an action of  the associated
permutations, and it is therefore far from being free. By contrast, if we
take for~$\SS$ the braid group~$\BB_\infty$ with $\op$ defined by $\xx \op \yy = \xx \,
\sh{}(\yy) \, \sig1 \, \sh{}(\xx)\inv$, where we recall $\sh{}$ is the shift endomorphism of~$\BB_\infty$ that
maps~$\sig\ii$ to~$\sig{\ii+1}$ for each~$\ii$, and if we send $\xx_\pp$ to~$1$ (or to any other fixed
braid) for each~$\pp$, then the corresponding action~\eqref{E:LDAction} of~$\BP\infty$ on~$(\BB_\infty)^*$
is \emph{free}, in the sense that $\aa = \aa'$ holds whenever $\ss \act \aa = \ss \act \aa'$ holds for at least one
sequence~$\ss$ in~$(\BB_\infty)^*$: this follows from Lemma~III.1.10 of~\cite{Dhr}. This suggests that the
associated category~$\CAct{\BP\infty}{(\BB_\infty)^*)}$ has a very rich structure.

\section*{Appendix: Other algebraic laws}

The above approach of self-distributivity can be developed for other algebraic laws as well. However, at least
from the viewpoint of Garside structures, the case of self-distributivity seems quite particular.

\subsection*{The case of associativity}

Associativity is the law $\xx(\yy\zz) = (\xx\yy)\zz$. It is syntactically close to self-distributivity, the only
difference being that the variable~$\xx$ is not duplicated in the right hand side. Let us say that a
term~$\tt'$ is an $A$-expansion of another term~$\tt$ if $\tt'$ can be obtained from~$\tt$ by applying the
associativity law in the left-to-right direction only, \ie, by iteratively replacing subterms of the form $\tt_1 \op
(\tt_2 \op \tt_3)$ by the corresponding term $(\tt_1 \op \tt_2) \op \tt_3$. Then the counterpart of
Proposition~\ref{P:Confluence} is true, \ie, two terms~$\tt, \tt'$ are equivalent up to
associativity if and only if they admit a common $A$-expansio, a trivial result since every
size~$\nn$ term~$\tt$ admits as an $A$-expansion the term~$\f(\tt)$ obtained
from~$\tt$ by pushing all brackets to the left.

As in Sections~\ref{S:LCDm} and~\ref{S:MLD}, we can introduce the category~$\CAm$ whose oblects are terms,
and whose morphisms are pairs~$(\tt, \tt')$ with $\tt'$ an $A$-expansion of~$\tt$. As in
Section~\ref{S:Labelling}, we can take positions into account, using $\AA_\a$ when associativity is applied at
address~$\a$, and introduce a monoid~$\MA$ that describes the connections between the
generators~$\AA_\a$ \cite{Dhb}. Here the relations of Lemma~\ref{L:RLD} are to be replaced by
analogous new relations, among which the MacLane--Stasheff Pentagon relations
$\AA_\a^2 = \AA_{\a1} \AA_\a \AA_{\a0}$. The monoid~$\MA$ turns out to be a well known object: indeed,
it is (isomorphic to) the submonoid~$\Fp$ of R.~Thompson's group~$F$ generated by the standard
generators~$\xx_1, \xx_2, ...$ \cite{CFP}.

Finally, as in Section~\ref{S:CLD}, we can introduce the category~$\CA$, whose objects are terms, and whose
morphisms are triples $(\tt, \aa, \tt')$ with~$\aa$ in~$\MA$ and $\tt \act \aa = \tt'$. Using $\psi(\tt)$ for
the term obtained from~$\tt$ by pushing all brackets to the right, we have

\begin{propS}
\label{P:Assoc}
The categories~$\CAm$ and~$\CA$ are isomorphic; $\CAm$ is left-Garside with Garside map $\tt
\mapsto (\tt, \f(\tt))$, and right-Garside with Garside map~$\tt \mapsto (\psi(\tt), \tt)$.
\end{propS}

This result might appear promising. It is not! Indeed, the involved Garside
structure(s) is trivial: the maps~$\f$ and~$\psi$ are constant on each orbit of
the action of~$\MA$ on terms, and it easily follows that every morphism in~$\CAm$ and~$\CA$ is left-simple
and right-simple so that, for instance, the greedy normal form of any morphism always has
length one\footnote{We do not claim that the monoid~$\MA$ is
not an interesting object in itself: actually it is, with rich nontrivial algebraic and geometric
properties, see~\cite{Dhb}; we only say that the current Garside category approach is not relevant here.}. The
only observation worth noting is that
$\CA$ provides an example where the left- and the right-Garside structures are not compatible, and, therefore,
we have no Garside structure in the sense of Definition~\ref{D:Garside}.

\subsection*{Central duplication}

We conclude with still another example, namely the exotic \emph{central duplication} law
$\xx(\yy\zz) = (\xx\yy)(\yy\zz)$ of~\cite{Dgj}. The situation there turns out to be very similar to that of
self-distributivity, and a nontrivial left-Garside structure appears. As there is no known
connection between this law and other widely investigated objects like braids, it is probably not necessary to
go into details.

\end{document}